\documentclass[a4paper, 11pt]{amsart}
\def\doctype{}

\usepackage{latexsym,amssymb,amsthm}
\usepackage{color, dsfont, fancyhdr}
\usepackage{nicematrix}
\usepackage{tikz}
\usepackage{hyperref}
\usepackage{float}
\usepackage{enumitem}
\usepackage{graphicx}

\newcommand{\A}{\mathrm{A}}
\renewcommand\S{\mathrm{S}}

\newcommand\Z{\mathbb{Z}}
\newcommand\Q{\mathcal{Q}}

\newcommand{\cB}{\mathcal{B}}

\renewcommand\P{\mathcal{P}}
\newcommand\fix{\mathrm{fix}}
\newcommand\supp{\mathrm{supp}}

\newcommand{\comment}[1]{}
\newcommand\aug{\fboxsep=-\fboxrule\!\!\!\fbox{\strut}\!\!\!}

\fancypagestyle{titlepage}{
\fancyhead[R]{\doctype}
\fancyhead[CL]{}
\cfoot{\vspace{5pt} \thepage}
}

\let\oldsection\section
\newcommand\boldsection[1]{\oldsection{\bf #1}}
\newcommand\starsection[1]{\oldsection*{\bf #1}}
\makeatletter
\renewcommand\section{\@ifstar\starsection\boldsection}
\makeatother

\newtheorem{thm}{Theorem}[section]  
\newtheorem{lemma}[thm]{Lemma}     
\newtheorem{cor}[thm]{Corollary}
\newtheorem{conj}[thm]{Conjecture}

\newtheorem{prop}[thm]{Proposition}

\newtheoremstyle{algorithm}
  {11pt}		  
  {11pt}  
  {\tt}  
  {}     
  {\bf}  
  {. }    
  {\newline}    
  {}     
\theoremstyle{algorithm}
\newtheorem{alg}{Algorithm}[section]

\theoremstyle{definition}

\makeatletter
\renewcommand*\@maketitle{%
  \normalfont\normalsize
  \@adminfootnotes
  \@mkboth{\@nx\shortauthors}{\@nx\shorttitle}%
  \global\topskip0\p@\relax 
  \@settitle
  \ifx\@empty\authors \else {\vskip 1em
\vtop{\centering\shortauthors\@@par}} \fi
  \ifx\@empty\@date \else {\vskip 1em \vtop{\centering\@date\@@par}}\fi 
  \ifx\@empty\@dedicatory
  \else
    \baselineskip18\p@
    \vtop{\centering{\footnotesize\itshape\@dedicatory\@@par}%
      \global\dimen@i\prevdepth}\prevdepth\dimen@i
  \fi
  \@setabstract
  \normalsize
  \if@titlepage
    \newpage
  \else
    \dimen@34\p@ \advance\dimen@-\baselineskip
    \vskip\dimen@\relax
  \fi
} 
\renewcommand*\@adminfootnotes{%
  \let\@makefnmark\relax  \let\@thefnmark\relax
  \ifx\@empty\@subjclass\else \@footnotetext{\@setsubjclass}\fi
  \ifx\@empty\@keywords\else \@footnotetext{\@setkeywords}\fi
  \ifx\@empty\thankses\else \@footnotetext{%
    \def\par{\let\par\@par}\@setthanks}%
  \fi
\thispagestyle{titlepage}
}
\makeatother

\pgfsetlayers{nodelayer,edgelayer} 

\begin{document}

\title[]{\large On Cameron's greedy conjecture}

\author{Coen del Valle and Colva M. Roney-Dougal}
\address{
School of Mathematics and Statistics,
University of St Andrews, St Andrews, UK
}
\email{cdv1@st-andrews.ac.uk, Colva.Roney-Dougal@st-andrews.ac.uk}

\thanks{Both authors are grateful to Peter Cameron for many useful and interesting conversations. Research of Coen~del~Valle is supported by the Natural Sciences and Engineering Research Council of Canada (NSERC), [funding reference number PGSD-577816-2023], as well as a University of St Andrews School of Mathematics and Statistics Scholarship.}

\date{\today}

\begin{abstract}
A base for a permutation group $G$ acting on a set $\Omega$ is a subset $\mathcal{B}$ of $\Omega$ whose pointwise stabiliser $G_{(\mathcal{B})}$ is trivial.  There is a natural greedy algorithm for constructing a base of relatively small size. We write $\mathcal{G}(G)$ the maximum size of a base it produces, and $b(G)$ for the size of the smallest base for $G$. In 1999, Peter Cameron conjectured that there  exists an absolute constant $c$ such that every finite primitive group $G$ satisfies $\mathcal{G}(G)\leq cb(G)$. We show that if $G$ is $\S_n$ or $\A_n$ acting primitively then either Cameron's Greedy Conjecture holds for $G$, or $G$ falls into one class of possible exceptions.
\end{abstract}

\maketitle

\vspace{-0.15cm}
\section{Introduction}
A \emph{base} for a permutation group $G$ acting on a finite set $\Omega$ is a subset $\cB = \{\alpha_1,\alpha_2,\dots, \alpha_k\}\subseteq\Omega$ whose pointwise stabiliser $G_{(\cB)}$ is trivial. The size $b(G)$ of a smallest base for $G$ is the \emph{base size} of $G$. Bases rose to prominence in the 1970s through the study of permutation group algorithms (see e.g.~\cite{ser}); for many algorithms the complexity is a function of the base size.  Blaha~\cite{blaha} shows that the problem of determining the base size of an arbitrary finite permutation group is NP-hard. 

We can derive elementary upper and lower bounds as follows. If $\{\alpha_i\}_{i=1}^{k}$ is a base for $G$ then each $g\in G$ is uniquely determined by the tuple $(\alpha_i^g)_{i=1}^{k}$, so $|G|\leq |\Omega|^{k}$. On the other hand, if $\{\alpha_i\}_{i=1}^{k}$ is \emph{irredundant} --- that is $G_{\alpha_1,\alpha_2,\dots,\alpha_i}>G_{\alpha_1,\alpha_2,\dots,\alpha_{i+1}}$ for all $i$ --- then each index is at least two, so $|G|\geq 2^k$, and hence $(\log |G|)/(\log |\Omega|)\leq k\leq \log |G|$. (All of our logarithms are to the base two, unless otherwise specified.)

One can compute a relatively small base using a greedy algorithm, as analysed by Blaha~\cite{blaha}. Any point of $\Omega$ with a smallest stabiliser lies in a largest orbit. Hence, if we let $\mathcal{B}_0=\emptyset$ and $\mathcal{B}_i=\mathcal{B}_{i-1}\cup\{\beta_i\}$, $i>0$, where $\beta_i\in\Omega$ is a point in any largest orbit of $G_{(\mathcal{B}_{i-1})}$, then there is an $n\geq0$ such that $\mathcal{B}_n$ is an irredundant base --- called a \emph{greedy base} --- for $G$ which produces the greatest reduction in group size at each stage. 

Let $\mathcal{G}(G)$ denote the maximum size of a greedy base for $G$. In~\cite{blaha}, Blaha shows that the greedy approximation is nearly sharp, and in particular that there is some absolute constant $d$ such that $$\mathcal{G}(G)\leq db(G)\log \log |\Omega|.$$
Moreover, for any $k\geq 2$ and $n$ sufficiently large, Blaha constructs (intransitive) groups of degree $n$ satisfying $b(G)=k$ and $\mathcal{G}(G)\geq\frac{1}{5}k\log\log n$.

For primitive groups the situation appears different, as illustrated by this 1999 conjecture of Peter Cameron.

\begin{conj}[Cameron's Greedy Conjecture~\cite{cam}]\label{cc}
There is some absolute constant $c$ such that if $G$ is a finite primitive permutation group then $\mathcal{G}(G)\leq cb(G)$.
\end{conj}

Let $I(G)$ be the maximum size of an irredundant base for $G$. A greedy base is irredundant, and so $b(G)\leq \mathcal{G}(G)\leq I(G)$. However, the gap between $b(G)$ and $I(G)$ can be arbitrarily large, even for primitive groups. For example, in~\cite{peiran} it is shown that for a primitive action of the symmetric group $\mathrm{S}_n$ with point stabiliser $H\ne \A_n$ primitive on $[n]:=\{1,2,\dots, n\}$, an upper bound of $I(\mathrm{S}_n)\leq O(\sqrt{n})$ is best possible, and yet~\cite{burn2} shows that in this case $b(\mathrm{S}_n)=2$ for $n\geq 13$. Therefore Cameron's Greedy Conjecture suggests that $\mathcal{G}(G)$ should fall much closer to $b(G)$ than to $I(G)$. This can be seen explicitly in recent work of the first author~\cite{sporadic}, which shows that for $G$ an almost simple primitive sporadic group, $\mathcal{G}(G)=b(G)$, whereas work of Lee~\cite{Lee} shows that in the vast majority of such cases $I(G)>b(G)$. Similarly, recent work of Brenner and the authors~\cite{soluble} has shown that if $G$ is primitive of odd order then $\mathcal{G}(G)=b(G)\leq 3$.

In this paper we consider the almost simple primitive groups $G$ with alternating socle. Burness, Guralnick, and Saxl~\cite{burn2} show that for such $G$ one of the following holds:
\begin{itemize}
    \item[(i)] $G$ is acting on $\binom{[n]}{r}$, the collection of $r$-subsets of $[n]$ with $2r<n$; 
    \item[(ii)] $G$ is acting on partitions of $[kl]$ into $k$ parts of size $l$ with $kl=n$;  
    \item[(iii)] $b(G) = 2$; or
    \item[(iv)] $G$ falls into a  list of 18 exceptions: see \cite[Corollaries 1.4 and 1.5 and Remark 1.7]{burn2}.
\end{itemize}

 In 2021, Morris and Spiga~\cite{ms} building on~\cite{bgl} gave explicit formulae for $b(G)$ for all $(k,l)$ pairs with partition action, and the base size of the actions on $[n]\choose r$ has recently been completely determined by the authors~\cite{dvrd} and independently for $\S_n$ in~\cite{MeSp}.

 In this paper we make progress towards a proof of Conjecture~\ref{cc}. Let $\mathcal{F}$ be a family of finite permutation groups up to permutation isomorphism. We say that $\mathcal{F}$ is \emph{ravenous} if there exists some constant $c>0$ such that $\mathcal{G}(G)\leq cb(G)$ for all $G\in\mathcal{F}$; if we are given an explicit constant $c$ we say that $\mathcal{F}$ is \emph{$c$-ravenous}. Restated, Cameron's Greedy Conjecture is that the family of finite primitive permutation groups is ravenous. 

\begin{thm}\label{main}
Let $\mathcal{F}$ be the family of almost simple primitive groups with alternating socle, excluding those acting on ${[n]\choose r}$ with $4r^2>n>2r$. Then $\mathcal{F}$ is an $11$-ravenous family.
\end{thm}
No significant effort has been made to optimise the value 11 --- it seems likely that it should be less than $2$.

 In Section~\ref{sec:sn} we prove Theorem~\ref{mainsnr}, which states that $\S_n$ and $\A_n$ acting on ${[n]\choose r}$ with $n\geq4r^2$ are 17/10-ravenous. In Section \ref{sec:parts} we prove Theorem~\ref{camskl}, which states that $\S_n$ and $\A_n$ acting on uniform partitions are 11-ravenous. Theorem~\ref{main} then follows from the observation that if $G$ is transitive and $b(G) = 2$ then $\mathcal{G}(G) = 2$, and explicit computations with the 18 groups in (iv) above.

\section{$\S_n$ and $\A_n$ on $r$-subsets}\label{sec:sn}

Let $n>2r \geq 4$. Let $\S_{n,r}$ and $\A_{n,r}$ be $\S_n$ and $\A_n$, respectively, in their actions on $\Omega:= {[n]\choose r}$. For the majority of this section we analyse $\S_{n,r}$, before deducing an analogous result for $\mathrm{A}_{n,r}$. We shall show that $\mathcal{F}_1:=\{\S_{n,r},\A_{n,r} : n\geq 4r^2\}$ is $17/10$-ravenous: see Theorem~\ref{mainsnr}.

We start with a condition on sets of $r$-sets which is equivalent to being a base. Let $\mathcal{B}$ be a set of $r$-subsets of $[n]$. For $u\in [n]$, define the \emph{neighbourhood} of $u$ to be $N_\mathcal{B}(u):=\{\alpha \in\mathcal{B} : u\in \alpha\}$. Then $\mathcal{B}$ is a base for $\S_{n,r}$ if and only if $N_\mathcal{B}(u)\ne N_\mathcal{B}(v)$ for all distinct $u,v\in [n]$. 

Fix $G=\mathrm{S}_{n,r}$, let $\alpha_i \in \Omega$ be the $i$th $r$-set chosen by the greedy algorithm, let $\mathcal{B}_i=\{\alpha_1,\alpha_2,\dots, \alpha_i\}$, and let $G_{i}=G_{(\mathcal{B}_i)}$. Note that in its action on $[n]$ the group $G_i$ is a direct product of symmetric groups on its orbits and so, up to $G$-conjugacy, $G_i$ is uniquely determined by the sizes of these orbits. 
The $G_i$-orbit size of each $\alpha \in \Omega$ is a product of binomial coefficients given by the size of the  intersection of $\alpha$ with each $G_i$-orbit on $[n]$. Therefore, we can think of the greedy algorithm at step $i+1$ as making a sequence of $r$ choices of points via the following meta-greedy algorithm. 
\begin{alg}[{\tt MetaGreedy}\text{}]

\noindent{Input:}
A collection $\mathcal{B}_i:=\{\alpha_1,\alpha_2,\dots, \alpha_i\}$ of $r$-subsets of $[n]$.\\
\noindent{Step $0$:}
Set $M_0:=\emptyset$.\\
\noindent{Step $1\leq j\leq r$:}
Let $m_j\in [n]\setminus M_{j-1}$ be any point from any $G_{i}$-orbit $\Delta\subseteq[n]$ that \linebreak maximises
\begin{equation}
    \label{star}
\frac{|\Delta\setminus M_{j-1}|}{|\Delta\cap M_{j-1}|+1},
\end{equation}
and let $M_j=M_{j-1}\cup\{m_j\}$. If \eqref{star} has maximum value 1 then in addition, choose $m_j$ to be a point with smallest $\mathcal{B}_i$-neighbourhood.\\
\noindent{Output:}
$\alpha_{i+1}:=M_r$.

\end{alg}

\begin{lemma}
Let $\alpha$ be an $r$-set. Then $|\alpha^{G_i}|$ is maximal amongst all elements of $\Omega$ if and only if there is a possible output $\alpha_{i+1}$ of {\tt MetaGreedy$(\mathcal{B}_i)$} such that $(G_i)_\alpha\cong (G_i)_{\alpha_{i+1}}$.
\end{lemma}
\begin{proof}
Let $\mathcal{O}$ be the set of $G_i$-orbits on $[n]$. First, suppose that $|\alpha^{G_i}|$ is maximal so that
$$\left|\alpha^{G_i}\right|=\prod_{\Delta\in\mathcal{O}}{|\Delta|\choose |\Delta\cap \alpha|}\geq \prod_{\Delta\in\mathcal{O}}{|\Delta|\choose |\Delta\cap \beta|},$$ 
for all $\beta \in \Omega$. We expand the binomial coefficients to obtain the inequality $$\left|\alpha^{G_i}\right|=\prod_{\Delta\in\mathcal{O}}\prod_{l=0}^{|\Delta\cap \alpha|-1}\frac{|\Delta|-l}{l+1}\geq \prod_{\Delta\in\mathcal{O}}\prod_{l=0}^{|\Delta\cap \beta|-1}\frac{|\Delta|-l}{l+1}.$$ 

Consider  {\tt MetaGreedy$(\mathcal{B}_i)$}. The set $\alpha$ contains a point $m_1$ from a largest $G_i$-orbit on $[n]$, so {\tt MetaGreedy$(\mathcal{B}_i)$} can choose $M_1 \subseteq \alpha$. Suppose inductively that $M_{j}\subseteq \alpha$ for some $j\geq 1$. Note $$|M_{j}^{G_i}|=\prod_{\Delta\in\mathcal{O}}\prod_{l=0}^{|\Delta\cap M_{j}|-1}\frac{|\Delta|-l}{l+1},$$ so that $$|\alpha^{G_i}|=|M_{j}^{G_i}|\cdot \prod_{\Delta\in\mathcal{O}}\prod_{l=|\Delta\cap M_{j}|}^{|\Delta\cap \alpha|-1}\frac{|\Delta|-l}{l+1}.$$ 
Since $|\alpha^{G_i}|$ is maximal,  $\alpha\setminus M_{j}$ contains a point from at least one orbit $\Delta$ that maximises \eqref{star}. Therefore if $\frac{|\Delta|-|\Delta\cap M_{j}|}{|\Delta\cap M_{j}|+1}\ne 1$ then {\tt MetaGreedy$(\mathcal{B}_i)$} can pick $m_{j+1}\in \Delta\setminus M_{j}$. If this holds for all $j$ then by induction {\tt MetaGreedy$(\mathcal{B}_i)$} can return $\alpha_{i+1} = M_{r} = \alpha$. 
Otherwise there exists a $j$ for which \eqref{star} has maximum value $1$. Since this maximum value is non-increasing in $j$, and $n>2r$, the maximum value of~\eqref{star} is 1 for all $s\geq j$. It follows that {\tt MetaGreedy$(\mathcal{B}_i)$} chooses each $m_{s} \in [n]$ to have minimum neighbourhood size. Since each stabiliser is a product of symmetric groups on its orbits we deduce that $(G_i)_{M_{j-1}}\cong (G_i)_{M_{j}}\cong (G_i)_\alpha$.  Hence in both cases there is a permissible output ${\alpha_{i+1}}$ of {\tt MetaGreedy$(\mathcal{B}_i)$} such that $(G_i)_{\alpha}\cong (G_i)_{\alpha_{i+1}}$.

For the other direction it suffices to show that all outputs of {\tt MetaGreedy$(\mathcal{B}_i)$} yield $G_i$-orbits of the same size. Whenever \eqref{star} has maximum value $o$ attained by multiple orbits, {\tt MetaGreedy} will successively choose a point from each such orbit, and $|M_j^{G_i}|=o\cdot|M_{j-1}^{G_i}|$, so the result follows. 
\end{proof}

Thus without loss of generality, we may assume each $\alpha_i$ in a greedy base is chosen by {\tt MetaGreedy}. Comparisons of the value of \eqref{star} are therefore useful. It will be most helpful in the form given now. 

\begin{lemma}\label{orbcomp}
Let $\Delta_1$ and $\Delta_2$ be any two $G_i$-orbits on $[n]$. Then $$\frac{|\Delta_1\setminus \alpha_{i+1}|+1}{|\Delta_1\cap \alpha_{i+1}|}\geq \frac{|\Delta_2\setminus \alpha_{i+1}|}{|\Delta_2\cap \alpha_{i+1}|+1}.$$
\end{lemma}
\begin{proof}
Consider {\tt MetaGreedy$(\mathcal{B}_i)$}. Let $j$ be maximal subject to $m_{j}\in \Delta_1$. Then $$\frac{|\Delta_1\setminus M_{j-1}|}{|\Delta_1\cap M_{j-1}|+1}\geq \frac{|\Delta_2\setminus M_{j-1}|}{|\Delta_2\cap M_{j-1}|+1}\geq \frac{|\Delta_2\setminus \alpha_{i+1}|}{|\Delta_2\cap \alpha_{i+1}|+1}.$$ 
The maximality of $j$ implies $\Delta_1\cap \alpha_{i+1}=(\Delta_1\cap M_{j-1})\cup\{m_{j}\}$, and so $|\Delta_1\setminus M_{j-1}|=|\Delta_1\setminus \alpha_{i+1}|+1$, and $|\Delta_1\cap M_{j-1}|+1=|\Delta_1\cap \alpha_{i+1}|$, thus the result follows.
\end{proof}

Let $s$ be maximal subject to $\alpha_1,\alpha_2,\dots, \alpha_s$ being disjoint, let $E_i=[n]\setminus\left(\bigcup_{j=1}^i \alpha_j\right)$ be the points in no set in $\mathcal{B}_i$, and let $e_i=|E_i|$.
Conversely, let $F_i=\{m\in [n] : |N_{\mathcal{B}_i}(m)|>1\}$ be the set of points in more than one set in $\mathcal{B}_i$. If two elements of $\mathcal{B}_i$ intersect in a single point $m \in [n]$, then $m$ is fixed by $G_j$ for all $j \ge i$; we shall prove in Lemma~\ref{p>3r} 
that in many cases the pairwise intersections of elements of $\mathcal{B}_i$ have size at most 1, and hence for most $i \in \{1, \ldots,\mathcal{G}(G)\}$ the set $F_i$ consists of all points fixed by $G_i$ with neighbourhoods of
size at least two. 
Eventually we shall prove that in fact all points in $F_i$ are fixed by $G_j$ for $j \ge i$.

We shall use Lemma~\ref{orbcomp} to show that for $n\geq 4r^2$ and any $i$, there is a large collection of $G_i$-orbits on $[n]$ all of size differing by at most one. 
For $i \in \{1, \ldots, s\}$ let $\mathcal{U}_i=\mathcal{B}_i$. For $i \ge s+1$, recursively define $\mathcal{U}_{i}:=\{A\setminus F_{i}  \, : \, A\in \mathcal{U}_{i-1}\},$  so that for $i \ge s+1$  the set $\mathcal{U}_{i}$ consists of the subsets of each of $\alpha_1,\dots,\alpha_s$ containing all points belonging to a unique element of $\mathcal{B}_i$. Thus, $\mathcal{U}_i$ is a collection of $G_i$-orbits in its natural action on $\alpha_1\sqcup\cdots \sqcup \alpha_s\subseteq[n]$. For $u \in \{0, \ldots, r\}$ 
 and $i\in\{1,2,\dots,\mathcal{G}(G)\}$, let 
$$\mathcal{O}_{u,i}=\{\alpha_j\setminus F_i  \, : \, j \le i, \ |\alpha_j \setminus F_i |\in\{u,u-1\}\},$$ 
which we shall view as a multiset, just to allow for multiple occurrences of the empty set.
Then $\mathcal{O}_{u,i}$ is the multiset of all $G_i$-orbits of size  $u$ or $u-1$ consisting of points chosen exactly once by step $i$. Now  for $i>1$ let $u_i$ be the largest $u \in \{0, \ldots, r\}$ satisfying 
\begin{enumerate}
    \item[(i)] $\mathcal{U}_i\subseteq \mathcal{O}_{u,i}$; and
    \item[(ii)] $\mathcal{O}_{u,i}$ contains at least one $u$-set other than $\alpha_i\setminus F_i$,
\end{enumerate}
or undefined if no such $u$ exists.
The next result shows that $u_i$ exists for all $i \ge 2$.

\begin{lemma}\label{p>3r}
Suppose $n\geq 4r^2\geq 16$. Then the following all hold.
\begin{enumerate}
\item[\emph{(i)}]
$s\geq 3r$. 
\item[\emph{(ii)}] $u_2=r$, and $u_i\in\{u_{i-1},u_{i-1}-1\}$ for all $i>2$, so $u_i$ exists for all $i$. Furthermore, {\tt Metagreedy}$(\mathcal{B}_i)$ never chooses more than one point from any set in $\mathcal{O}_{u_i,i}$.
\item[\emph{(iii)}] $|\mathcal{O}_{u_i, i}| \ge 3r$ for all $i \ge s$. 
\item[\emph{(iv)}] Suppose $u_{i}\geq 2$ and $|\alpha\setminus F_i|\leq u_i$ for all $\alpha\in\mathcal{B}_i$. If every element of $F_{i}$ is fixed by $G_{i}$ and has a neighbourhood of size exactly two, then the same is true for $F_{i+1}$ and $G_{i+1}$.
\end{enumerate}
\end{lemma}
\begin{proof}
By definition of $s$, there exists an $i \le s$ such that
$\alpha_i \cap \alpha_{s+1} \neq \emptyset$, so Lemma~\ref{orbcomp} with $(\Delta_1, \Delta_2) = (\alpha_{i}, E_s)$  shows that $$\frac{r-1+1}{1}\geq \frac{(n-sr)-(r-1)}{(r-1)+1}.$$ Therefore $s\geq (n-r^2-(r-1))/r\geq (3r^2-(r-1))/r>3r-1$, but $s$ is an integer hence (i) holds.

We show (ii) inductively. For $2\leq i\leq s$ by definition $\mathcal{U}_i=\mathcal{B}_i = \mathcal{O}_{r,i}$ so $u_i = r$ and $F_i$ is empty. Let $i>s$. By induction, there is some $u_{i-1}$ such that all $s$ sets in $\mathcal{U}_{i-1}$ have size in $\{u_{i-1},u_{i-1}-1\}$, so by (i) the group $G_{i-1}$ has at least $3r$ orbits on $[n]$ of size either $u_{i-1}$ or $u_{i-1}-1$. Now, $(u_{i-1}-1)/2\leq (u_{i-1}-1)/1$ with equality if and only if $u_{i-1}=1$ so from \eqref{star} we deduce the ``furthermore" statement.  Moreover, {\tt MetaGreedy$(\mathcal{B}_i)$} only chooses a point from an orbit of size $u_{i-1}-1$ if it has already chosen a point from every orbit of size $u_{i-1}$, so the result follows. 

Parts (iii) and (iv) now follow from Parts (i) and (ii).
\end{proof}

Since Lemma~\ref{p>3r}(ii) shows that $u_i$ is uniquely determined by $i$, we shall write $\mathcal{O}_i$ instead of $\mathcal{O}_{u_i,i}$. We say that an  $\alpha \in\mathcal{B}_i$ is \emph{excessive} if $|\alpha \setminus F_i|>u_i$, and a collection of $r$-sets is \emph{excessive} if at least one member is excessive. 
If we never encounter excessive sets, then our greedy bases are well-behaved.

\begin{prop}\label{prop:noexcess}
Suppose $n\geq 4r^2 \ge 16$, and let $\mathcal{B}$ be a greedy base for $\mathrm{S}_{n,r}$. If no $\mathcal{B}_i$ is excessive, then $|\mathcal{B}|\leq \frac{2n}{r}$.
\end{prop}
\begin{proof}
     Let $i$ be maximal such that $u_i=2$. It follows inductively from Lemma~\ref{p>3r}(iv) that every element of $F_{i+1}$ is fixed by $G_{i+1}$ and has a neighbourhood of size two. Since $u_{i+1}=1$, by Lemma~\ref{p>3r}(ii) and (iii) there are at least $2r$ singleton sets in $\mathcal{O}_{i+1}$, and at most $r$ copies of the empty set. Since $u_{i+1}=1$, and $\mathcal{B}_{i+1}$ is not excessive,~\eqref{star} implies that $e_{i} = |E_i| \leq 3$, and hence $e_{i+1}\leq 2$. We deduce that $\mathcal{B}_{i+2}$ is a base with no neighbourhoods of size greater than two. 
     
     Thus double counting pairs $(\alpha,a)$ such that $a\in\alpha\in\mathcal{B}$ shows that $$|\mathcal{B}|r=\sum_{a\in [n]}|N_{\mathcal{B}}(a)|\leq 2n,$$ hence the result.
\end{proof}

We now consider the possibility that some $\mathcal{B}_i$ is excessive.

\begin{lemma}\label{ez}
Suppose $n\geq 4r^2 \ge 16$. Suppose that $\mathcal{B}_i$ is not excessive, but $\mathcal{B}_{i+1}$ is excessive. 
Then $\alpha_{i+1}$ is the unique excessive set in $\mathcal{B}_{i+1}$. 
Let $f= |\alpha_{i+1}\cap E_i|$ be the number of points chosen for the first time in $\alpha_{i+1}$. Then either
\begin{enumerate}
\item[\emph{(i)}]
$u_{i+1}=u_i=f-1$, $\mathcal{O}_{i+1}$ contains at least $r$ sets of size $u_i$, and $e_{i+1}\leq u_{i}^2+u_i$; or
\item[\emph{(ii)}]
$u_{i+1}+1=u_i\leq f\leq u_i+1$ and $e_{i+1}\leq u_i^2$.
\end{enumerate}
\end{lemma}
\begin{proof}
By Lemma~\ref{p>3r}(ii) for $j \le i$ we can bound $|\alpha_j \cap F_i| \le |\alpha_j \cap F_{i+1}| + 1$, and if $|\alpha_j \setminus F_i| = u_i-1$ then $\alpha_{i+1} \cap \alpha_j \neq \emptyset$ only if $u_{i+1}=u_i-1$, whence $\alpha_{i+1}$ is the unique excessive set.

 Since $\mathcal{B}_i$ is not excessive, $|\alpha_i \setminus F_i| \le u_i<r$. Thus
there exists a $G_{i-1}$-orbit $\Delta$ of size $u_{i-1}$ such that $|\Delta\cap\alpha_i|\geq 1$. Now Lemma~\ref{orbcomp} with $(\Delta_1, \Delta_2) = (\Delta, E_{i-1})$ shows that 
\begin{equation}\label{eq1}
u_{i-1} \geq \frac{e_{i}}{e_{i-1}-e_i+1} \geq \frac{e_{i-1}-u_{i}}{u_{i}+1}
\end{equation}

Suppose that $u_{i+1}=u_i$. Lemma~\ref{p>3r}(ii) states that $u_{i-1}\leq u_i+1$, so working with the first and last expression in~\eqref{eq1} bounds  $e_{i-1}\leq u_i^2+3u_i+1$. Furthermore, the first inequality in \eqref{eq1} yields $$e_i\leq\frac{(e_{i-1}+1)(u_i+1)}{u_i+2}\leq \frac{(u_i^2+3u_i+2)(u_i+1)}{u_i+2}=u_i^2+2u_i+1.$$ 
Additionally, since $u_{i+1}=u_i$ there is a $u_i$-set $\Delta \in \mathcal{O}_i$ such that $|\Delta \cap \alpha_{i+1}| = 0$.
Therefore,  Lemma~\ref{orbcomp}  with $(\Delta_1,\Delta_2)=(E_i, \Delta)$ yields 
\begin{equation}\label{eq2}
u_i\leq \frac{e_i-f+1}{f},
\end{equation} 
so $f\leq (e_i+1)/(u_i+1)$. Since $e_i\leq u_i^2+2u_i+1$ it follows that $f\leq (u_i^2+2u_i+2)/(u_i+1)<u_i+2$. Moreover, $\alpha_{i+1}$ is excessive in $\mathcal{B}_{i+1}$, so $f>u_i$, whence $f=u_i+1$ proving our first claim. We deduce immediately that $e_{i+1}=e_i-f\leq u_i^2+u_i$, proving our third.

Since $e_i<e_{i-1}$, substituting for $f$ into \eqref{eq2} and combining with \eqref{eq1} demonstrates that $u_i<(e_{i-1}-u_i)/(u_i+1) \leq u_{i-1}.$ But $u_{i-1}\leq u_i+1$ by Lemma~\ref{p>3r}(ii), hence $u_{i-1}=u_i+1$. Since $|\mathcal{O}_{i}|\geq 3r$ and $u_{i-1}=u_i+1$, we deduce from Lemma~\ref{p>3r}(iii) that $\mathcal{O}_{i+1}$ contains at least $r$ sets of size $u_i$, as desired, hence (i) holds.

By Lemma~\ref{p>3r}(ii) we may now assume that $u_{i+1}=u_i-1$. Therefore there are more than $2r$ sets in $\mathcal{O}_{i+1}$ of size $u_{i+1}$, hence $u_{i-1}=u_i$ and there exists $\Delta 
\in \mathcal{O}_{i+1}$ of size $u_{i}-1$ with $|\Delta \cap \alpha_{i+1}| =  0$.  
By~\eqref{eq1}, $e_{i-1}\leq u_i^2+2u_i$, and $(u_i+1)e_i\leq u_i(e_{i-1}+1)$, whence $$e_i\leq \frac{u_i(u_i^2+2u_i+1)}{u_i+1}=u_i^2+u_i.$$ Now Lemma~\ref{orbcomp} with $(\Delta_1, \Delta_2) = (E_i, \Delta)$ shows that
$$u_i-1\leq\frac{e_i-f+1}{f}\leq \frac{u_i^2+u_i+1-f}{f},$$ 
hence $f\leq \left\lfloor \frac{u_i^2+u_i+1}{u_i}\right\rfloor =u_i+1$. Finally,  $\mathcal{B}_{i+1}$ is excessive so $f\geq u_i$, hence $e_{i+1}\leq u_i^2$ as required. 
\end{proof}

We now show that the greedy algorithm exhibits some self-correcting behaviour, so that the introduction of excessive sets does not introduce any neighbourhoods of size greater than two.

\begin{lemma}\label{a>4}
Suppose $n\geq 4r^2 \ge 16$. Suppose that $\mathcal{B}_i$ is not excessive, but $\mathcal{B}_{i+1}$ is excessive. If $u_i>2$ then $\mathcal{B}_{i+j}$ is not excessive for some $j \in \{2, 3\}$. 
Moreover, if every element of $F_{i+1}$ is fixed by $G_{i+1}$ and has a neighbourhood of size exactly two, then the same is true for $F_{i+j}$ and $G_{i+j}$. 
\end{lemma}

\begin{proof}
By Lemma~\ref{ez}, $\alpha_{i+1}$ is the unique excessive set in $\mathcal{B}_{i+1}$. Let $f= |\alpha_{i+1}\cap E_i|$ as in Lemma~\ref{ez}. 

Suppose first that $u_{i+1}=u_i$, so $f=u_i+1$. From $u_i>2$ we get {$(u_i+1-1)/2<(u_i-1)/1$}. Therefore, by ~\eqref{star}, $|\alpha_{i+2}\cap \alpha_j |\leq 1$ for all $j \leq i+1$. 
Additionally, Lemma~\ref{ez} tells us that $\mathcal{O}_{i+1}$ contains at least $r$ $G_{i+1}$-orbits of size $u_i$, 
so there exists at least one such $G_{i+1}$-orbit $\Delta$ with $|\Delta \cap \alpha_{i+2}| = 0$. In particular, $u_{i+2} = u_i$.
In this case $e_{i+1} \le u_i^2 + u_i$, so 
Lemma~\ref{orbcomp} with $(\Delta_1, \Delta_2) = (E_{i+1}, \Delta)$ shows that
$$\frac{(u_i^2+u_i)-|E_{i+1}\cap \alpha_{i+2}|+1}{|E_{i+1}\cap \alpha_{i+2}|}\geq\frac{e_{i+1}-|E_{i+1}\cap \alpha_{i+2}|+1}{|E_{i+1}\cap \alpha_{i+2}|}\geq u_i.$$ Rearranging gives $u_i^2+u_i+1\geq (u_i+1)|E_{i+1}\cap \alpha_{i+2}|$, and so $|E_{i+1}\cap \alpha_{i+2}|\leq \lfloor u_i+(u_i+1)^{-1}\rfloor = u_i$, and therefore $\alpha_{i+2}$ is not excessive. Additionally, $\alpha_{i+1} \setminus F_{i+1}$ and $E_i$ are the largest $G_{i+1}$-orbits, so {\tt MetaGreedy$(\mathcal{B}_{i+1})$} chooses a point from $\alpha_{i+1}\setminus F_{i+1}$. Hence also $|\alpha_{i+1} \setminus F_{i+2}| \le u_{i+1} + 1 - 1 = u_{i+2}$, and therefore $\mathcal{B}_{i+2}$ is not excessive.

By Lemma~\ref{p>3r}(ii) we may now suppose that $u_{i+1}=u_i-1$.  We shall show that $\mathcal{B}_{i+3}$ is not excessive. It follows from Lemma~\ref{p>3r}(iii) that there are more than $2r$ sets in $\mathcal{O}_{i+1}$ of size $u_i-1 \geq 2$, so $|\alpha_{i+2}\cap \alpha_j |\leq 1$ for all $j \leq i+1$ by Lemma~\ref{p>3r}(ii), so there exist $u_{i+1}$-sets  $\Delta, \Gamma \in\mathcal{O}_{i+1}$ such that $|\Delta\cap\alpha_{i+2}|=0$ and $|\Gamma \cap \alpha_{i+2}| = 1$. The first of these implies that $u_{i+2} = u_{i+1}$.
Additionally, $e_{i+1}\leq u_i^2$ by Lemma~\ref{ez}, so Lemma~\ref{orbcomp} with $(\Delta_1,\Delta_2)=(E_{i+1},\Delta)$ yields 
$$\frac{u_i^2-|E_{i+1}\cap \alpha_{i+2}|+1}{|E_{i+1}\cap \alpha_{i+2}|}\geq\frac{e_{i+1}-|E_{i+1}\cap \alpha_{i+2}|+1}{|E_{i+1}\cap \alpha_{i+2}|}\geq u_i-1.$$ 
Thus, $|E_{i+1}\cap \alpha_{i+2}|\leq \lfloor u_i+u_i^{-1}\rfloor = u_i$.

Next, Lemma~\ref{orbcomp} with $(\Delta_1, \Delta_2) = (\Gamma, E_{i+1})$ bounds $u_i-1 \ge e_{i+2}/(e_{i+1}-e_{i+2}+1)$, and rearranging gives $e_{i+2}\leq(u_i-1)(e_{i+1}+1)/u_i$. Since $e_{i+1}\leq u_i^2$, it now follows that $$e_{i+2}\leq \left\lfloor \frac{(u_i-1)(u_i^2+1)}{u_i} \right\rfloor =
\left \lfloor \frac{u_i^3-u_i^2+u_i-1}{u_i} \right \rfloor =u_i^2-u_i.$$ Now, $|\alpha_{i+1}\setminus F_{i+2}|<|\alpha_{i+1}\setminus F_{i+1}|=u_i+1 = u_{i+2} + 2$, so $\alpha_{i+1}$ and $\alpha_{i+2}$ are the only possible excessive sets in $\mathcal{B}_{i+2}$, and each has at most $u_{i+2} + 1$ points not in $F_{i+2}$. 

There are more than $r$ sets in $\mathcal{O}_{i+2}$ of size $u_{i+2} \geq 2$, so $|\alpha_{i+3}\cap \alpha_{j}|\leq 1$ for all $j \le i+2$ by Lemma~\ref{p>3r}(ii), and moreover there is a $u_{i+2}$-set $\Theta\in\mathcal{O}_{i+2}$ disjoint from $\alpha_{i+3}$. Thus $u_{i+3} = u_{i+2} = u_{i+1} = u_i - 1$. 
Combining $e_{i+2} \leq u_i^2 - u_i=u_{i+3}^2+u_{i+3}$ with Lemma~\ref{orbcomp} applied to $(\Delta_1,\Delta_2)=(E_{i+2},\Theta)$ gives 
$$\frac{(u_{i+3}^2+u_{i+3})-|E_{i+2}\cap \alpha_{i+3}|+1}{|E_{i+2}\cap \alpha_{i+3}|}\geq \frac{e_{i+2}-|E_{i+2}\cap \alpha_{i+3}|+1}{|E_{i+2}\cap \alpha_{i+3}|}\geq u_{i+3},$$ 
hence $|E_{i+2}\cap \alpha_{i+3}|\leq u_{i+3}$. We have now shown that $\alpha_{i+1}$, $\alpha_{i+2}$, and $\alpha_{i+3}$ each have at most $u_{i+3}$ points outside of $F_{i+3}$. Therefore, there are no excessive sets in $\mathcal{B}_{i+3}$.
\end{proof}

\begin{thm}
Suppose $n\geq 4r^2$. Then $\mathcal{G}(\S_{n,r})\leq \frac{2n}{r}+1$.
\end{thm}

\begin{proof}
If $r = 1$ this is clear, so assume that $r \ge 2$.
Let $\mathcal{B}$ be a greedy base for $\S_{n,r}$. By Proposition~\ref{prop:noexcess}, without loss of generality we may let $i$ be the final index such that $\mathcal{B}_{i+1}$ is excessive whilst $\mathcal{B}_{i}$ is not. If $u_i>2$, then by applying Lemma~\ref{p>3r} and Lemma~\ref{a>4} inductively we deduce that there is some $j\in\{2,3\}$ such that every element of $F_{i+j}$ is fixed by $G_{i+j}$ and has a neighbourhood of size exactly two. We now argue as in the proof of Proposition~\ref{prop:noexcess} to deduce that $|\mathcal{B}|\leq 2n/r$.

On the other hand suppose that $u_i\leq 2$. By Lemma~\ref{ez}, $\alpha_{i+1}$ is the unique excessive set in $\mathcal{B}_{i+1}$. Let $f=|\alpha_{i+1}\cap E_i|$ be the number of newly chosen points. We first show that 
at most $r$ points from $[n]$ have neighbourhoods of size $3$ in $\mathcal{B}$, and all others have neighbourhoods of size at most two. From this it will follow that $$|\mathcal{B}|r=\sum_{a\in [n]}|N_\mathcal{B}(a)|\leq 2(n-r)+3r=2n+r,$$ yielding the desired result.

Suppose first that $u_{i+1}=u_i$. By Lemma~\ref{ez}, $e_{i+1}\leq u_i^2+u_i$, and $f=u_i+1$. If $u_i=1$ then $e_{i+1}\leq f=2$, and so $G_{i+1}\cong \mathrm{Sym}(E_{i}\cap \alpha_{i+1})\times\mathrm{Sym}(E_{i+1})$. By Lemma~\ref{ez}, $\mathcal{O}_{i+1}$ contains at least $r$ sets each consisting of a single point with a $\mathcal{B}_{i+1}$-neighbourhood of size one, thus --- since {\tt MetaGreedy} chooses points with smallest neighbourhood --- $\mathcal{B}_{i+2}$ is a base with no points in neighbourhoods of size greater than two. If instead $u_i=2$, then $f=3$ and $e_{i+1}\leq 6$ --- one can verify for each possible $e_{i+1} \in \{3, 4, 5, 6\}$ that $\mathcal{B}$ puts at most $r$ points in neighbourhoods of size three.

On the other hand suppose $u_{i+1}=u_i-1$. By Lemma~\ref{ez}, $u_i\leq f\leq u_i+1$ and $e_{i+1}\leq u_i^2$. If $u_i=1$, then $f\leq 2$, and $e_{i+1}\leq 1$, so $|G_{i+1}|\leq2$, hence $\mathcal{B}_{i+2}$ is a base with at most $r$ points in $\mathcal{B}_{i+2}$-neighbourhoods of size three. On the other hand, suppose $u_i=2$. Then $2\leq f\leq 3$ and $e_{i+1}\leq 4$. Additionally, $u_{i+1}=1$ so $\mathcal{O}_{i+1}$ contains at least $2r$ fixed points in neighbourhoods of size one --- the greedy algorithm terminates within the next two steps if $r>2$, and three steps if $r=2$, either way, no neighbourhoods of size three are introduced.
\end{proof}

Since any pointwise stabiliser for the action of $\S_{n,r}$ is a direct product of symmetric groups, and any non-trivial symmetric group contains a transposition, we immediately get the following corollary.
\begin{cor}\label{ansub}
Suppose $n\geq 4r^2$. Then $\mathcal{G}(\mathrm{A}_{n,r})\leq \frac{2n}{r}+1$.
\end{cor}
\begin{proof}
For $\S_{n,r}$ there is a transposition in each stabiliser until we reach a base; the orbits for $\mathrm{A}_{n,r}$ therefore have the same size as the corresponding orbits for $\mathrm{S}_{n,r}$ and so the greedy algorithm will make the same choices for $\mathrm{A}_{n,r}$ as it does for $\mathrm{S}_{n,r}$,  but may stop one step earlier.
\end{proof}

With that we deduce that the subset actions with $n$ sufficiently large form a ravenous family.

\begin{thm}\label{mainsnr}
 Let $G\in\{\mathrm{A}_{n,r},\mathrm{S}_{n,r}\}$ with $n\geq 4r^2$. Then $\mathcal{G}(G)\leq \frac{17}{10}b(G).$
\end{thm}
\begin{proof}
If $r = 1$ then this is clear. For $r \ge 2$,
the ratio $\left(\frac{2n}{r}+1\right)/b(G)$ is maximized by setting $G=\mathrm{A}_{n,r}$ with $r=2$ and  $n=4\cdot 2^2=16$. In this case~\cite[Theorem 1.1]{dvrd} gives $b(\mathrm{A}_{n,r})=10$, while Corollary~\ref{ansub} gives $\mathcal{G}(G)\leq 17$, hence the result.
\end{proof}

\begin{rk}
Experimentally, we find that the greedy algorithm actually does much better than this. 
For example, for the extremal case $\A_{16,2}$, every greedy base is minimum.
\end{rk}

\section{$\S_n$ and $\A_n$ on partitions}\label{sec:parts}
We now analyse the greedy algorithm applied to the actions of $\S_{kl}$ and $\mathrm{A}_{kl}$ on the set $\Pi = \Pi_{k, l}$ of all \emph{$(k,l)$-partitions:} partitions of $[kl]$ into exactly $k$ parts of size $l$. Thus the stabiliser of a point is $\S_l \wr \S_k$. We denote the corresponding permutation groups by $\S_{k\times l}$ and $\mathrm{A}_{k\times l}$.

Throughout this section we use a recent paper by Peter Cameron and the authors~\cite{cdvrd} in which the 2-point stabilisers of the greedy algorithm are studied. In most cases, understanding the structure of the 2-point stabilisers will be sufficient for determining the greedy base size (see Lemma~\ref{key}). We start by considering the case $l=2$ --- here every greedy base is minimum.

\begin{prop}\label{l=2}
Let $G\in\{\A_{k\times 2},\S_{k\times 2} : k\geq 3\}$. Then $$\mathcal{G}(G)=\begin{cases}4&\text{ if $G=\S_{3\times 2}$,}\\
3&\text{ otherwise.}\end{cases}$$
\end{prop}
\begin{proof}
We first consider $G = \S_{k \times 2}$. We check computationally that $\mathcal{G}(\S_{3\times 2})=4$, so suppose that $k\geq 4$. By \cite[Proposition 3.1]{cdvrd} the smallest 2-point stabiliser is the dihedral group $\mathrm{D}_{2k}$ of order $2k$. By the proof of ~\cite[Proposition 3.1]{cdvrd}, without loss of generality we may let $$\P_1=\{\{1,2\},\{3,4\},\dots, \{2k-1,2k\}\} \mbox{ and } \P_2=\{\{2k,1\},\{2,3\},\dots,\{2k-2,2k-1\}\}$$ so that $G_{\P_1,\P_2}=\mathrm{D}_{2k}.$ Let $$\P_3=\{\{1,2\},\{3,5\},\{4,2k\}\}\cup\{\{2i,2i+1\} : 3\leq i\leq k-1\},$$ which exists since $k\geq 4$, and let $\tau\in G_{\P_1,\P_2,\P_3}$. Then $\{1,2\}$ is the only set common to $\P_1$ and $\P_3$ so $\{1,2\}^\tau=\{1,2\}$. If $(1,2)^\tau=(2,1)$ then since $\P_2^\tau=\P_2$, we deduce that $(3,2k)^\tau=(2k,3)$. Next, $(4,2k-1)^{\tau}=(2k-1,4)$ since $\P_1^\tau=\P_1$, but then $\{4, 2k\}^\tau=\{2k-1, 3\}$ which is not in $\P_3$, a contradiction. Thus $(1,2)^\tau=(1,2)$. This forces $3^\tau=3$ since $\P_2^\tau=\P_2$, which implies that $4^\tau=4$ since $\P_1^\tau=\P_1$. Continuing this way, we see that $\tau=1$, so $\mathcal{G}(G) = 3$.

Now let $G = \A_{k \times 2}$. We check computationally that $\mathcal{G}(G)=3$ for $k\in\{3, 4\}$, so suppose $k\geq 5$. If $k$ is odd, then by \cite[Proposition 3.1]{cdvrd} the smallest 2-point stabiliser of $\S_{k\times 2}$ is $\mathrm{D}_{2k}$, which contains an odd permutation. Thus, the greedy algorithm makes identical choices for $G$ as for $\S_{k\times 2}$ and the result follows from the result for $\S_{k \times 2}$.

Thus without loss of generality $k$ is even. The unique (up to isomorphism) 2-point stabiliser of $\mathrm{S}_{k\times 2}$ of size $2k$, namely $\mathrm{D}_{2k}$, is a subgroup of $\mathrm{A}_{2k}$. On the other hand, by \cite[Proposition 3.1]{cdvrd} the (unique up to isomorphism) second smallest 2-point stabiliser in $\S_{k \times 2}$ is $2\times \mathrm{D}_{2k-2}$ of order $4k-4$, and hence since $k\geq 5$ the greedy algorithm chooses the first two partitions $\Q_1$ and $\Q_2$ so that $(\A_{k\times 2})_{\Q_1,\Q_2}=\mathrm{D}_{2k-2}$. By the proof of ~\cite[Proposition 3.1]{cdvrd}, without loss of generality 
\[\Q_1=\{\{1,2\},\{3,4\},\dots,\{2k-1,2k\}\} \mbox{ and } \Q_2 =\{\{2i,2i+1\}: 1\leq i\leq k-2 \}\cup\{\{2k-2,1\},\{2k-1,2k\}\}.\] Define $$\Q_3=\{\{1,2\},\{3,5\},\{4,2k-2\},\{2k-1,2k\}\}\cup\{\{2i,2i+1\} : 3\leq i\leq k-2\}.$$  Then $G_{\Q_1,\Q_2,\Q_3}$ is trivial by an identical argument as for $\P_1,\P_2,\P_3$.
\end{proof}

We split the rest of this section into two subsections. In the first we develop some useful theory, in the second we prove Theorem~\ref{camskl}.

\subsection{Intersection arrays and their properties}\label{machinery}
In this subsection we recall some key definitions from \cite{cdvrd}, and generalise the concept of intersection matrices.

For the following to be well-defined we will usually think of each partition in $\Pi$ as being ordered by the smallest element in each part. Let $t\in\{1,2,3\}$, and let $a_1,\dots,a_t\in\mathbb{N}$. Denote by $\mathcal{M}_{a_1,\dots,a_t}$ the set of all $a_1\times\cdots\times a_t$ arrays with entries in $\mathbb{Z}_{\geq 0}$. The \emph{intersection array} $M(\P_1,\dots,P_t)$ of $\mathcal{P}_1,\dots, \P_t \in \Pi$  is the array $M=(m_{i_{1}\cdots i_t})\in\mathcal{M}_{k,\dots,k}$ with $m_{i_1\cdots i_t}=\left|\bigcap_{j=1}^tP_{ji_j}\right|$. We also say that $M$ is a $(k,l)$-\emph{intersection array} or a $(k,l)$-intersection $t$-array.

The next result is a 3-dimensional analogue of \cite[Lemma 2.1]{cdvrd}.

\begin{lemma}\label{matchar} Let $W=(w_{ijs})\in\mathcal{M}_{k,k,k}$ and let  $N= M(\P,\Q)=(n_{ij})$ for some $\P,\Q \in \Pi$. Then there is a $\mathcal{T} \in \Pi$ such that $W=M(\P,\Q,\mathcal{T})$ if and only if $\sum_{s=1}^k w_{ijs}=n_{ij}$ for all $i,j$ and $\sum_{1\leq i,j\leq k} w_{ijs}=l$ for all $s$.
\end{lemma}
\begin{proof}
First suppose that $\mathcal{T} \in \Pi$ satisfies $W=M(\mathcal{P},\mathcal{Q},\mathcal{T})$. For all $i,j$ $$\sum_{s=1}^k w_{ijs}=|P_{i}\cap Q_j|,$$ hence $N=M(\P,\Q)$ satisfies $n_{ij}=\sum_{s=1}^k w_{ijs}$. Additionally, for each $s$ $$\sum_{i,j=1}^k w_{ijs}=\sum_{i,j=1}^k|P_{i}\cap Q_{j}\cap T_s|=\sum_{j=1}^k|Q_j\cap T_s|=|T_s|=l.$$

On the other hand, suppose instead $\sum_{s=1}^k w_{ijs}=n_{ij}$ for all $i,j$ and $\sum_{1\leq i,j\leq k} w_{ijs}=l$ for all $s$. For $1\leq s\leq k$, let $T_{s}=\bigsqcup_{1\leq i,j\leq k} X_{ijs}$ where $X_{ijs}$ is any set of size $w_{ijs}$ such that $$X_{ijs}\subseteq (P_{i}\cap Q_{j})\setminus\left(\bigcup_{r<s} T_{r}\right).$$ This is possible since $\sum_{s=1}^k w_{ijs}=n_{ij}=|P_{i}\cap Q_{j}|$. Then $$|T_{s}|=\sum_{1\leq i,j\leq k} |X_{ijs}|=\sum_{1\leq i,j\leq k}w_{ijs}=l,$$ and so $\mathcal{T}:=\{T_{1},T_{2},\dots,T_{k}\}$ is a $(k,l)$-partition. Moreover, since $\P$ and $\Q$ are partitions, $P_i\cap Q_j$ is disjoint from $P_{i'}\cap Q_{j'}$ for any $(i',j')\ne (i,j)$, so in particular $|P_{i}\cap Q_{j}\cap T_{s}|=|P_{i}\cap Q_{j}\cap X_{ijs}|=|X_{ijs}|=w_{ijs}$, as desired. 
\end{proof}

We shall say that the $(k,l)$-intersection array $W=(w_{ijs})$ \emph{arises} from $N=(n_{ij})$ if $\sum_{s=1}^k w_{ijs}=n_{ij}$ for all $i,j\leq k$.

We now consider how to determine up to $\S_{kl}$-conjugacy the stabiliser of a triple of $(k,l)$-partitions from their intersection array, generalising~\cite[Lemma 2.2]{cdvrd}. The group $\S_{a}\times\S_{b} (\times\S_{c})$ acts coordinatewise on $\mathcal{M}_{a,b(,c)}$. We write $K$ to denote this group when $a$ and $b$ (and $c$) are clear. Let $G=\S_{k\times l}$, and $H\leq G_\mathcal{P}$ for some $\P \in \Pi$.  We say that $\mathcal{P}$ is \emph{partwise fixed} by $H$ if the induced action of $H$ on the parts of $\mathcal{P}$ is trivial. An array $W$ is \emph{partwise fixed} if $K_W=1$.

The proof of the next lemma is analogous to the 2-dimensional case in~\cite{cdvrd}.

\begin{lemma}
\label{top}
Let  $W=M(\P,\Q,\mathcal{T})$ for some distinct $\P, \Q, \mathcal{T} \in \Pi$. Then $$H:=G_{\P,\Q,\mathcal{T}}\cong\left(\prod_{i_1,i_2, i_3\in [k]}\mathrm{Sym}\left( P_{i_1} \cap Q_{i_2} \cap T_{i_3} \right)\right): K_{W}.$$ In particular, $|H|=|K_{W}|\prod_{i_1,i_2,i_3}(n_{i_1i_2 i_3}!)$, and $H\cong G_{\P_1,\Q_1, \mathcal{T}_1}$ 
 whenever $M(\P_1,\Q_1,\mathcal{T}_1)=W$. Moreover, $\P$, $\Q$ and $\mathcal{T}$ are partwise fixed by $H$ if and only if $K_N=1$.
\end{lemma}

Given $v = (v_1, \ldots, v_k) \in \Z^k$ define the \emph{cyclic-shift matrix} $\theta(v)$ of $v$ to be the $k \times k$ matrix with $(i,j)$-entry $v_{j-i+1}$, interpreting subscripts modulo k. If $\sum v_i=l$ then $\theta(v)$ is a $(k,l)$-intersection array with $|K_{\theta(v)}|\geq k$.

Let $s$ and $t$ be positive integers. Consider a sequence of \emph{multiplicities} $m_1,m_2,\dots,m_d$ with $\sum_{i = 1}^d m_i<t$, and a sequence of \emph{ranks} of the form $$r_1>r_2>\cdots> r_d>\lceil s/t\rceil\quad\text{ or }\quad r_1<r_2<\cdots< r_d<\lfloor s/t \rfloor,$$ such that $\sum_{i = 1}^d m_ir_i<s$.  Let $X=[(m_1,r_1),(m_2,r_2),\dots, (m_d,r_d)]$ 
 --- we call $X$ a \emph{multiplicity sequence}. We use double curly braces and exponential notation for multisets, and write $\omega_A(x)$ to denote the multiplicity of $x$ in the multiset $A$. If $A_0 = \{\{r_1^{m_1}, r_2^{m_2}, \ldots, r_d^{m_d}\}\}$, so that $\omega_A(r_i) = m_i$ for all $i$, then $A_0$ can be completed to a multiset $A$ of size $t$ with sum $s$ by adding only elements equal to either the floor or ceiling of the average required value, which is $$x=x(s,t,X):=\frac{s-\left(\sum m_ir_i\right)}{t-\sum m_i}.$$ 
Let $b=b(s,t,X)\leq t-\sum m_{i}$ be a positive integer satisfying $$s-\sum m_ir_i=b\lceil x\rceil +\left(t-\left(\sum m_i\right)-b\right)\lfloor x\rfloor,$$ and let $S=S(s,t,X)$ be the collection of non-negative integer multisets $A$ of size $t$ with sum $s$ such that for each $i \in \{1, \ldots, d\}$, if $r_i> \lceil s/t\rceil$ then $A$ has at least $\sum_{j=1}^i m_j$ entries greater than or equal to $r_i$, and otherwise $A$ has at least   $\sum_{j=1}^i m_j$ entries less than or equal to $r_i$.  

For an array $W$, let $W^*$ be the multiset of entries of $W$ and let $\omega_W(x)=\omega_{W^*}(x)$ for each $x$.

\begin{lemma}[{\cite[Lemma 2.4]{cdvrd}}]\label{minN}
With the above notation, $\min\left\{\prod_{a\in A} a! : A\in S\right\}$ is achieved uniquely by the integer multiset $B\in S$ with $\omega_B(r_i)=m_i$, and all other elements of $B$ as close to each other as possible. Thus, this minimum is precisely $$\left(\prod_{i=1}^{d}r_i!^{m_i}\right)\lceil x\rceil!^b\lfloor x\rfloor!^{t-\left(\sum m_i\right)-b}.$$ 
\end{lemma}
 
For $\P_1,\P_2,\dots,\P_m \in \Pi$ we define $\mathcal{S}_{\P_1,\P_2,\dots,\P_m}$ to be 
the multiset of all sizes of $m$-wise intersections across parts of the $\P_i$. When clear from context we shall write $G$ for $\S_{k\times l}$ and $\mathcal{S}_{m}=\mathcal{S}_{\P_1,\P_2,\dots,\P_m}$.
Since $G$ is transitive on $\Pi$ we will often think of the greedy algorithm as first choosing an intersection 2-array with stabiliser of minimal size.
For intersection arrays $N=M(\mathcal{P},\Q)$ and $W = M(\P,\Q,\mathcal{T})$, we write $G(N)$ and $G(W)$ for $G_{\mathcal{P},\Q}$ and  $G_{\mathcal{P},\Q,\mathcal{T}}$.

We now have all of the preliminary theory to state and prove a key lemma.

\begin{lemma}\label{key}
Let $G=\S_{k\times l}$ and let $\P_1,\P_2,\dots,\P_m \in \Pi$. Suppose each $\P_i$ is partwise fixed by $G_{\P_1,\P_2,\dots,\P_m}$ and that $\omega_{\mathcal{S}_m}(t_m)\geq k$ for some $t_m\not\equiv 0\pmod{k}$. Then $$\mathcal{G}((A_{k\times l})_{\P_1,\P_2,\dots,\P_m})\leq\mathcal{G}(G_{\P_1,\P_2,\dots,\P_m})=\lceil \log_k(\max(\mathcal{S}_m))\rceil.$$
\end{lemma}
By saying that a $(k,l)$-intersection array $N$ satisfies the conditions of this lemma, we shall mean that any set of partitions for which $N$ is the intersection array satisfies these conditions. 
\begin{proof}
We show the equality for $G$, and then deduce the result for $\A_{k \times l}$. We will show inductively that for all $s\geq 0$ with $m+s\leq \mathcal{G}(G)$ the greedy algorithm chooses a partition $\mathcal{P}_{m+s} \in \Pi$ such that each of the following conditions hold:
\begin{enumerate}[label=(\roman*), ref=\roman*]
\item\label{trivtop} $\mathcal{P}_{m+s}$ is partwise fixed by $G_{\mathcal{P}_1,\dots,\mathcal{P}_{m+s}}$;
\item\label{sizes} $\omega_{\mathcal{S}_{m+s}}(t_{m+s})\geq k$ for some $t_{m+s}\not\equiv 0\pmod{k}$; and
\item\label{szcn} $\max(\mathcal{S}_{m+s})=\lceil\max(\mathcal{S}_{m+s-1})/k\rceil$ if $s\geq 1$.
\end{enumerate}
By assumption, $\P_m$ satisfies the above conditions for $s = 0$, so the base case holds.
Suppose all conditions hold for some $s$, and let $\mathcal{A}=\{\{A_1,A_2,\dots, A_{k^{m+s}}\}\}$ be the multiset of all $(m+s)$-wise intersections of $\mathcal{P}_1,\dots,\mathcal{P}_{m+s}$, so that $\mathcal{S}_{m+s}=\{\{|A_i|:1\leq i\leq k^{m+s}\}\}$. 

Let $H=G_{\mathcal{P}_1,\dots,\mathcal{P}_{m+s}}$. 
Then for all $\P \in \Pi$,$$|H_{\P}|\geq\prod_{(A,P)\in\mathcal{A}\times\P} |A\cap P|!$$ 
with equality exactly when $\P$ is partwise fixed by $H_{\P}$. Such a product is minimised when $|A\cap P|\in\{\lceil |A|/k\rceil, \lfloor|A|/k\rfloor\}$ for all $(A,P)\in\mathcal{A}\times\P$, so if there exists a $\P\in\Pi$ that is partwise fixed by $H_{\P}$ and satisfies $|A\cap P|\in\{\lceil |A|/k\rceil, \lfloor|A|/k\rfloor\}$ for all $(A,P)\in\mathcal{A}\times\P$ then such a $\mathcal{P}$ will be chosen as $\P_{m+s+1}$.

By Condition (\ref{sizes}), there is some $t:=t_{m+s}$ such that at least $k$ elements  of $\mathcal{A}$, say $A_1,A_2,\dots, A_k$, have size $t$. 
Since $k \nmid t$ there exists a unique positive integer $a$ such that $t=a\lceil t/k\rceil +(k-a)\lfloor t/k \rfloor$. Consider the cyclic-shift matrix $N=\theta(v)$ where $v$ is the $k$-tuple with first $a$ entries $\lceil t/k\rceil $ and remaining entries $\lfloor t/k\rfloor$. The rows of $N$ are distinct, as are the columns, and moreover, each row and column sums to $t$. By \cite[Lemma 2.1]{cdvrd} there is a partition $\mathcal{B}=\{B_1,\dots, B_k\}$ of $\bigcup_{i=1}^k A_i$ into parts of size $t$ such that $N=M(\{A_1,\dots, A_k\},\mathcal{B})$. Let $\mathcal{C}=\{C_1,\dots,C_k\}$ be any partition of $[kl]\setminus \bigcup_{i=1}^k A_i$ into $k$ parts of size $l-t$ such that $|C_i\cap A_j|\in\{\lfloor |A_j|/k\rfloor,\lceil |A_j|/k\rceil\}$ for all $i\leq k<j$. Finally, 
let $\mathcal{P}_{m+s+1} =\{B_i\sqcup C_i : 1\leq i\leq k\}\in \Pi$.

By Condition (\ref{trivtop}), $A_i^H=A_i$ for all $i$. Let $\sigma\in H_{\mathcal{P}_{m+s+1}}$ map $P_j^\sigma = P_r$ for some $P_j,P_r\in\mathcal{P}_{m+s+1}$. Then $(A_i\cap P_j)^\sigma=A_i\cap P_r$ for all $i$, hence 
$$(n_{ij})_{1\leq i\leq k} = (|A_i\cap B_{j}|)_{1\leq i\leq k}=(|A_i\cap P_j|)_{1\leq i\leq k}=(|A_i\cap P_r|)_{1\leq i\leq k}=(|A_i\cap B_{r}|)_{1\leq i\leq k} = (n_{ir})_{1\leq i\leq k}.$$ 
But the columns of $N$ are distinct, hence $j=r$. Therefore, $\mathcal{P}_{m+s+1}$ satisfies Condition (\ref{trivtop}). Moreover, at least one of $\lceil t/k\rceil, \lfloor t/k\rfloor$ is not divisible by $k$, and both appear at least $k$ times in $N$. Since the entries of $N$ describe a submultiset of $\mathcal{S}_{m+s+1}$ Condition (\ref{sizes}) is satisfied. Finally, Condition (\ref{szcn}) holds as $|A\cap P|\in\{\lceil |A|/k\rceil, \lfloor|A|/k\rfloor\}$ for all $A$ in $\mathcal{A}$ and $P\in\P_{m+s+1}$. Therefore, the inductive claim holds.

The greedy algorithm terminates at the first $r$ such that all $(m+r)$-wise intersections have size at most 1, and these intersections are decreasing (at worst) by the recurrence $a_{i+1}=\lceil a_i/k\rceil$. A short calculation shows that $\mathcal{G}(G_{\P_1,\P_2,\dots,\P_m})=\lceil \log_k(\max(\mathcal{S}_m))\rceil$, as desired.

The above induction follows identically for $\A_{k\times l}$, since each $\S_{k\times l}$-stabiliser contains a transposition until the process terminates. However, for $\A_{k\times l}$ the greedy algorithm halts at the first  $r$ such that there is at most one $(m+r)$-wise intersection of size 2, and all others of size at most 1.
\end{proof}

The next lemma simplifies some logarithms from Lemma~\ref{key}.

\begin{lemma}\label{logfacts}
Suppose that $k\geq 7$ and $l=qk+\varepsilon$ for $\varepsilon\in\{0,\pm1,\pm2\}$.  Then $$\lceil\log_k(q+1)\rceil=\lceil \log_k(l+3)\rceil -1.$$
\end{lemma}
\begin{proof}
First note that 
\begin{equation*}
\lceil\log_k \lceil l/k\rceil\rceil=\lceil\log_k l\rceil -1.\end{equation*} 
If $\varepsilon>0$ then $\lceil\log_k (q+1)\rceil= \lceil\log_k \lceil l/k\rceil\rceil=\lceil\log_k l\rceil -1$. Since $k\geq 7$, this is equal to $\lceil \log_k(l+3)\rceil -1$. 
On the other hand, if $\varepsilon\leq 0$, then $\lceil\log_k (q+1)\rceil=\lceil\log_k (\lceil l/k\rceil+1)\rceil$, where $\lceil\log_k (\lceil l/k\rceil+1)\rceil=\lceil\log_k \lceil l/k\rceil\rceil+1$ if and only if $\lceil l/k\rceil$ is a power of $k$. 
Hence from $\lceil\log_k \lceil l/k\rceil\rceil=\lceil\log_k l\rceil -1$ we deduce that 
$$\lceil\log_k (q+1)\rceil=
\begin{cases} 
\lceil\log_k l\rceil&\text{if $l-\varepsilon$ is a power of $k$}\\ 
\lceil\log_k l\rceil-1&\text{otherwise.}
\end{cases}
 \ =\lceil\log_k (l+3)\rceil-1.$$\end{proof}

We finish with a lemma giving sufficient conditions for the existence of a  greedy base of size three.

\begin{lemma}\label{trivstab}
Suppose $k\geq 7$  and that $N^\ast =M(\P,\Q)^\ast \subseteq \{0,1,2\}$, and let $I =\{(i,j) : n_{ij}=2\}$. Suppose the following hold: 
\begin{itemize}
\item[{\rm(i)}]$K_N=\langle \sigma\rangle$, for some involution $\sigma$ with $t$ transpositions;
\item[\rm{(ii)}] there is some $(i,j)\in\supp(\sigma) \subseteq [k] \times [k]$ with $n_{ij}=1$;
\item[\rm{(iii)}] $|I|\leq 2l-3$; and
\item[\rm{(iv)}] $|I|+t\leq \frac{kl}{2}-1$.
\end{itemize} Then there exists a $(k,l)$-partition $\mathcal{T} \in \Pi$ such that $G_{\P,\Q,\mathcal{T}}=1$.
\end{lemma}

\begin{proof}
 For $(i,j)\in\supp(\sigma)$, if $n_{ij}=1$ let $P_i\cap Q_j = \{a_{(i,j)}\}$ 
 and if $n_{ij}=2$ let $P_i\cap Q_j=\{b_{(i,j)},c_{(i,j)}\}$. Let $J_1$ be a subset of 
 $\supp(\sigma)$ consisting of exactly one of $(i,j)$ or $(i,j)^\sigma$ for each 
 $(i,j)\in \supp(\sigma)$ with $n_{ij}=1$, and let $J_2 \subseteq\supp(\sigma)$ be defined in the same way
 but for $n_{ij}=2$. Then
$$G_{\P,\Q}=\left\langle\{(b_{(i,j)}\,c_{(i,j)}) : (i,j)\in I\}\cup \left\{\left(\prod_{(i,j)\in J_1}(a_{(i,j)}\,a_{(i,j)^\sigma})\right)\left(\prod_{(i,j)\in J_2}(b_{(i,j)}\,b_{(i,j)^\sigma})(c_{(i,j)}\,c_{(i,j)^\sigma})\right)\right\}\right\rangle.$$ 

We now see that $|\fix_{[kl]}(G_{\P,\Q})|=kl-2|I|-2|J_1|$, but $|J_1|\leq t$, hence \begin{equation}\label{fixeq}
|\fix_{[kl]}(G_{\P,\Q})|+|J_1|\geq kl-2|I|-t\geq kl-|I|-\left(\frac{kl}{2}-1\right)=\frac{kl}{2}-|I|+1\end{equation} 
by assumption (iv). In particular, from $|J_1|\leq t$ we deduce that 
$$|\fix_{[kl]}(G_{\P,\Q})|\geq 1+\frac{kl}{2}-|I|-t\geq 2,$$ 
and since $k\geq 7$ we also deduce from~\eqref{fixeq} that $|\fix_{[kl]}(G_{\P,\Q})|+ |J_1|> 2l-|I|$.

Construct a $2l$-set $T \subseteq [kl]$ as follows. For each $(i,j)\in I$ include $b_{i,j}$, and for one $(i,j)\in J_1$ include $a_{(i,j)}$ (such an $(i,j)$ exists by assumption (ii)). Next, include at least two fixed points of $G_{\P,\Q}$, so that $|T|$ is now $2+1+|I|\leq 2l$ by assumption (iii). Finally, since $|\fix_{[kl]}(G_{\P,\Q})|+|J_1|+|I|> 2l$ we may add arbitrary elements of $\fix_{[kl]}(G_{\P,\Q})$ and $\{a_{(i,j)} : (i,j)\in J_1\}$ to $T$ until  $|T|=2l$. Partition $T$ into $l$-sets $T_1$ and $T_2$ such that each contains a fixed point of $G_{\mathcal{P},\Q}$. Let $\mathcal{T}$ be any partition of $[kl]$ containing $T_1$ and $T_2$. Then since both $T_1$ and $T_2$ contain points fixed by $G_{\mathcal{P},\Q}$, it follows that both are fixed by $G_{\mathcal{P},\Q,\mathcal{T}}$. Finally,  $a_{(i,j)}\in T$ for some $(i,j)\in J_1$, whereas $a_{(i,j)^{\sigma}}\not\in T$, so $G_{\P,\Q,\mathcal{T}}\leq\langle(b_{(i,j)}\,c_{(i,j)}) : (i,j)\in I\rangle$. These generators have disjoint supports, and none fix both $T_1$ and $T_2$, hence $G_{\P,\Q,\mathcal{T}}=1$.
\end{proof}

\subsection{Proof of Theorem~\ref{camskl}}\label{many} 
Throughout, let $G = \S_{k \times l}$. We start with a similar set-up as in \cite{cdvrd} --- in particular, we suppose that $l>2$ and $k\ge 2$, and
write $l=qk+r$ with $-2\leq r\leq k-3$ whenever $k\geq 4$, and $0\leq r\leq k-1$ otherwise. We define an equivalence relation $\sim$ on the set of $(k,l)$-intersection arrays given by $A\sim B$ if and only if $A^*=B^*$ and $|K_A|= |K_B|$: notice that if $A \sim B$ then $|G(A)| = |G(B)|$, although the converse does not hold. Throughout the remaining proofs, statements of uniqueness shall be taken up to $\sim$.

\begin{prop}\label{logbase}
Suppose that either $k\geq 8$, or that $k = 7$ and $-2\leq r\leq 2$. Then $$\mathcal{G}(\A_{k\times l})\leq\mathcal{G}(G)=\lceil\log_k (l+3)\rceil+1.$$
\end{prop}
\begin{proof}
Let $\P$ and $\Q$ be the first two $(k,l)$-partitions chosen by the greedy algorithm. Suppose first that either $k\geq 8$ and $r\geq 3$, or that $k\geq7$ and $r=\pm2$. By \cite[Propositions 3.2 and 3.4]{cdvrd}, $\P$ and $\Q$ satisfy the hypotheses of Lemma~\ref{key} with largest intersection size $q+1$. Therefore, by Lemmas~\ref{key} and~\ref{logfacts} we deduce that $ \mathcal{G}(\A_{k \times l}) \leq \mathcal{G}(G)=\mathcal{G}(G_{\P,\Q})+2=\lceil \log_k (l+3)\rceil+1$, as desired.

 Suppose instead that $k\geq 7$ and $-1\leq r\leq 1$. Then by \cite[Propositions 3.6 and 3.8]{cdvrd}, either $\P$ and $\Q$ satisfy the conditions of Lemma~\ref{key} with greatest intersection size $q+1$, or $N=M(\P,\Q)$ satisfies the conditions of Lemma~\ref{trivstab} (see \cite[Proposition 3.6, Case $m=k-1$ and Proposition 3.8, Case $m=\lceil k/2\rceil-1$ (ii)]{cdvrd} for further details on the involution). In the first case 
 $$\mathcal{G}((\A_{k \times l})_{\P, \Q}) \leq \mathcal{G}(G_{\P,\Q})=\lceil\log_k (q+1)\rceil=\lceil\log_k (l+3)\rceil-1$$ 
 by Lemmas~\ref{key} and~\ref{logfacts}, and in the second $\mathcal{G}(\A_{k \times l}) \leq 
 \mathcal{G}(G)=3=\lceil\log_k (l+3)\rceil+1$ by Lemma~\ref{trivstab}.
\end{proof}

To finish our study of $\S_{k\times l}$ it remains to consider $k\leq 7$. Here we prove that the greedy algorithm constructs a small base for $l$ sufficiently large. Throughout, we represent $(k,l)$-intersection 3-arrays  as elements of $(\mathbb{Z}^k)^{k\times k}$.  

We start with $k=2$. We prove the following result by exhibiting the unique $(2,l)$-intersection array obtained from the first three choices of greedy, and then applying Lemma~\ref{key}. 
\begin{prop}\label{k=2}
Suppose $l\geq 10$. Then $$\mathcal{G}(\A_{2\times l})\leq\mathcal{G}(\S_{2\times l})\leq\lceil \log_2 (l+8)\rceil+1.$$
\end{prop}
\begin{proof}
A $(2,l)$-intersection 2-array $A$ is uniquely determined by its $(1,1)$ entry, and furthermore if $a_{11}=a_{12}$ then $|K_A|=4$, whilst $|K_A|=2$ otherwise. Consider $$N=\begin{bmatrix}\lfloor l/2\rfloor +1&\lceil l/2\rceil-1\\\lceil l/2\rceil-1&\lfloor l/2\rfloor+1\end{bmatrix}.$$ If $\max A^*\leq \lfloor l/2\rfloor$, then $A=qJ$, and so $|G(A)|=4(q!)^4=\frac{2q^2}{(q+1)^2}|G(N)|>|G(N)|$, since $q\geq 5$. On the other hand  by Lemma~\ref{minN}, $|G(N)| < |G(B)|$ for   all $(2,l)$-intersection 2-arrays $B$ with $\max B^*>\lfloor l/2\rfloor+1$. Therefore, $N$ is chosen by the greedy algorithm. We now determine the intersection 3-array $W$ chosen by the greedy algorithm.

\medskip

\noindent \underline{\textbf{Case $l\equiv 2 \pmod 4 $:}} Write $l = 4m + 2$ and let $$W=\begin{bmatrix}
   (m+2,m)&(m,m)\\
  (m-1,m+1)&(m+1,m+1)\end{bmatrix}.$$ Then $K_W=1$: we show that $W$ yields a 3-point stabiliser of minimum size. First, note that $w_{i2t}\in\left\{\lfloor n_{i2}/2\rfloor,\lceil n_{i2}/2\rfloor\right\}=\{m,m+1\}$. Additionally, by~Lemma~\ref{minN}, $w_{111}!w_{112}!=(m+2)!m!$ is minimum amongst all pairs $a!b!$ with $a + b = 2m+2$ and $\max\{a, b\} \ge m+2$, and $w_{211}!w_{212}!=(m-1)!(m+1)!$ is minimum amongst all such pairs with $a + b = 2m$ and $\max\{a, b\} \ge m+1$.  Moreover, there is no intersection $3$-array $V$ arising from $N$ which contains exactly one pair $v_{ij1}\ne v_{ij2}$, as then $\sum_{1\leq i,j\leq 2}v_{ij1}\ne \sum_{1\leq i,j\leq 2}v_{ij2}$. Therefore, there are exactly two arrays $V_1$ and $V_2$ distinct from $W$ such that $|G(V_i)/K_{V_i}|\leq|G(W)/K_W| = |G(W)|$ --- they are $$V_1=\begin{bmatrix}
        (m+2,m)&(m,m)\\
        (m,m)&(m,m+2)\\
    \end{bmatrix}\,\text{and }
    V_2=\begin{bmatrix}
        (m+1,m+1)&(m,m)\\
        (m,m)&(m+1,m+1)\\
    \end{bmatrix}
$$
Now, $K_{V_1}=\langle ((1\,2),(1\, 2),(1\,2))\rangle$, so that $|G(V_1)|=2\cdot(m+2)!^2m!^6=\frac{2m(m+2)}{(m+1)^2}|G(W)|>|G(W)|$, and $K_{V_2}=\langle ((1\,2),(1\, 2),1),(1,1,(1\,2))\rangle$, hence $|G(V_2)|=4\cdot (m+1)!^4m!^4=\frac{4m(m+1)}{(m+2)(m+1)}|G(W)|>|G(W)|$.
Therefore, $W$ uniquely yields a $3$-point stabiliser of minimum size.

\medskip 

\noindent \underline{\textbf{Case $l\not\equiv 2\pmod 4$:}}
Let $$W = \begin{bmatrix}(\lceil n_{11}/2\rceil+1,\lfloor n_{11}/2\rfloor-1)&(\lfloor n_{12}/2\rfloor,\lceil n_{12}/2\rceil)\\
(\lfloor n_{12}/2\rfloor,\lceil n_{12}/2\rceil)&(\lfloor n_{11}/2\rfloor,\lceil n_{11}/2\rceil)\end{bmatrix}.$$
Then $K_W = 1$: we shall show that $W$ is unique such that $|G(W)|$ is minimal. 

If $(i,j)\ne (1,1)$, then $w_{ijt}\in\left\{\lfloor n_{ij}/2\rfloor,\lceil n_{ij}/2\rfloor\right\}$. Moreover, by~Lemma~\ref{minN}, $w_{111}!w_{112}!$ is minimum amongst all $a!b!$ with $a+b= n_{11}$ and
$\max\{a, b\} \ge \lceil n_{11}/2\rceil+1$. Therefore, since $n_{11}$ is the largest entry of $N$ we deduce that any array $V\not\sim W$, arising from $N$ such that $|G(V)/K_V|\leq|G(W)/K_W| = |G(W)|$ has permutations of $(\lceil n_{11}/2\rceil,\lfloor n_{11}/2\rfloor)$ for $v_{11}$ and $v_{22}$, and permutations of $(\lceil n_{12}/2\rceil,\lfloor n_{12}/2\rfloor)$ for $v_{12}$ and $v_{21}$. 
Each such array $V$ satisfies $|K_V| \ge 2$, whence a short calculation yields $$|G(V)|\geq2(\lceil n_{11}/2\rceil)!^2(\lfloor n_{11}/2\rfloor)!^2(\lceil n_{12}/2\rceil)!^2(\lfloor n_{12}/2\rfloor)!^2=\frac{2\lfloor n_{11}/2\rfloor}{\lceil n_{11}/2\rceil+1}|G(W)|>|G(W)|,$$ since $n_{11}\geq 6$. Hence again $W$ uniquely yields a $3$-point stabiliser of minimum size.

\medskip

 Therefore, for all $l$, after three steps the greedy algorithm arrives at  $W$. In both cases there is an $x \in W^*$ such that $\omega_W(x),\omega_W(x+1)\geq 2$, and one of $x$ and $x+1$ is odd, so $W$ satisfies the conditions of Lemma~\ref{key}. Moreover, $\max W^*\leq \lfloor l/4\rfloor+2$, whence $$\mathcal{G}(A_{2\times l})\leq\mathcal{G}(G)\leq \lceil\log_2 (\lfloor l/4\rfloor +2)\rceil+3\leq\lceil\log_2(l+8)\rceil+1,$$ as desired.
\end{proof}

We split the case $k=3$ into Propositions~\ref{k=3} and \ref{other3}, depending on $l \pmod 3$. 

\begin{prop}\label{k=3}
Suppose $k=3$ and $l=3q$ for some $q\geq 5$. Then 
$$\mathcal{G}(\A_{k\times l})\leq\mathcal{G}(\S_{k\times l})\leq\lceil \log_3 (l+9)\rceil+1.$$
\end{prop}
\begin{proof}
We shall first exhibit a $(3,l)$-intersection 2-array $N$, then show that $N$ uniquely yields a stabiliser of minimum order and hence is chosen by the greedy algorithm. 

Let $$N=\begin{bmatrix}
q-2&q+2&q\\
q+2&q-1&q-1\\
q&q-1&q+1
\end{bmatrix}$$
be a $(3,3q)$-intersection array.
Each row and column has a distinct multiset of entries, so $K_N = 1$  by~\cite[Lemma 2.5]{cdvrd}, and $|G(N)|=(q+2)!^2(q+1)!q!^2(q-1)!^3(q-2)!$.

Let $A=(a_{ij})$ be a $(3,l)$-intersection 2-array minimising $|G(A)|$ with $a_{11}=\max A^*$. Throughout this proof all statements describing the structure of $A$ are to be taken up to permutation/transposition. If $a_{11}=q+1$, then every row and column of $A$ has multiset of entries one of $\{\{(q+1)^2,q-2\}\}$, $\{\{q^3\}\}$, or $\{\{q+1,q-1,q\}\}$, hence one may verify that $A$ is one of $$qJ,\begin{bmatrix} q+1&q-1&q\\
q-1&q&q+1\\
q&q+1&q-1\end{bmatrix},
\begin{bmatrix} q+1&q+1&q-2\\
q-1&q&q+1\\
q&q-1&q+1\end{bmatrix},
\begin{bmatrix} q+1&q+1&q-2\\
q+1&q-2&q+1\\
q-2&q+1&q+1\end{bmatrix},
\begin{bmatrix} q+1&q-1&q\\
q-1&q+1&q\\
q&q&q\end{bmatrix}.$$
Thus $|G(A)|$ is ${36(q!)^9},$ ${3(q+1)!^3q!^3(q-1)!^3},$ or similar expressions 
for the remaining three arrays. A straightforward calculation shows that in each case $|G(A)| \ge |G(N)|$ since $q\geq 5$. Thus if there exists an $A \neq N$ with $G(A) \le G(N)$ then $\max{A^\ast} > q+1$.

By~Lemma~\ref{minN}, if $A$ has at least three entries greater than $q+1$, then $|G(A)|\geq (q+2)!^3(q-1)!^6=\frac{(q-1)(q+2)}{q^2}|G(N)|>|G(N)|$. Hence $A$ has at most two entries greater than $q+1$. Moreover, if $a_{11}>q+2$, then the first row and column of $A$ both have entries less than $q-1$, and hence $A$ either has another entry which is at least $q+2$, or at least three entries $q+1$. Since $(q+3)!(q+2)!q!^2(q-1)!^5<(q+3)!(q+1)!^3(q-1)!^4(q-2)!$ we deduce from~Lemma~\ref{minN} that $$|G(A)|\geq (q+3)!(q+2)!q!^2(q-1)!^5
=\frac{(q+3)(q-1)}{q(q+1)}|G(N)|>|G(N)|.$$ Therefore, $\max A^\ast = a_{11}=q+2$, and $\omega_A(q+2) \in \{1, 2\}$.

Suppose first that $\omega_A(q+2)=1$. Then $\min A^\ast \ge q-2$  (for otherwise the row and column containing $\min A^\ast$ both have an entry at least $q+2$). Therefore, without loss of generality the first row and column of $A$ are $(q+2,q-2,q)$ or $(q+2,q-1, q-1)$, so $A$ is one of
$$\begin{bmatrix}q+2&q-2&q\\q-2&q+1&q+1\\q&q+1&q-1\end{bmatrix}\text{, } \begin{bmatrix}q+2&q-2&q\\q-1&q+1&q\\q-1&q+1&q\end{bmatrix}\text{, }\begin{bmatrix} q+2&q-1&q-1\\q-1&q+1&q\\q-1&q&q+1\end{bmatrix}.$$
In each case, a calculation shows that $|G(A)| > |G(N)|$. We conclude that  $A$ has exactly two entries greater than $q+1$, and these are both $q+2$.

The entries $q+2$ of $A$ are in distinct rows, for otherwise the third entry of the row is $q-4$, and the column containing $q-4$ must contain an entry greater than $q+1$, contrary to assumption. Therefore, the first row of $A$ is one of $(q+2,q-2,q)$, $(q+2,q-1,q-1)$, or $(q+2,q+1,q-3)$. A careful check yields the following possibilities for $A \neq N$:
$$\begin{array}{cll} \begin{bmatrix}
    q+2&q+1&q-3\\
    q-1&q&q+1\\
    q-1&q-1&q+2\\
\end{bmatrix}, & \begin{bmatrix}
    q+2&q+1&q-3\\
    q+1&q-2&q+1\\
    q-3&q+1&q+2\\
\end{bmatrix}, &
\begin{bmatrix}
    q+2&q-2&q\\
    q-2&q+2&q\\
    q&q&q\\
\end{bmatrix},
\\
\\ \begin{bmatrix}
    q+2&q+1&q-3\\
    q-2&q+1&q+1\\
    q&q-2&q+2\\
\end{bmatrix},& \begin{bmatrix}
    q+2&q+1&q-3\\
    q&q-1&q+1\\
    q-2&q&q+2\\
\end{bmatrix}.
\end{array}$$ 

For each such $A$ we may check that $|G(A)| > |G(N)|$.
Therefore $N$ is chosen by the greedy algorithm.

If $q-1$ is not divisible by $3$, then $N$ satisfies the conditions of Lemma~\ref{key},
so $\mathcal{G}(G(N))=\lceil\log_3(q+2)\rceil$. Therefore $$\mathcal{G}(G)=\lceil\log_3(q+2)\rceil+2=\lceil\log_3(3q+6)\rceil+1\leq\lceil\log_3(l+9)\rceil+1,$$ as required. 
Suppose instead that $q=3m+1$ for some $m$. Let $$W=
    \begin{bmatrix}
    (m-1,m,m)&(m+1,m+1,m+1)&(m+1,m,m)\\
    (m+1,m+1,m+1)&(m,m,m)&(m,m,m)\\
    (m,m,m+1)&(m,m,m)&(m+1,m+1,m)\\
\end{bmatrix}$$ 
be a $(3,l)$-intersection array arising from $N$. Since $N$ is partwise fixed it follows that $K_W$ stabilises the rows and columns of $W$, so acts diagonally as a subgroup of $\S_3$ on the entries. But $w_{11}$ and $w_{31}$ show that the first and last coordinate of each entry are fixed, so $W$ is partwise fixed. Moreover, the coordinates of each entry are such that the product of their factorials is minimised, and so the first three partitions chosen by the greedy algorithm will yield $W$. Finally, $W$ satisfies the conditions of Lemma~\ref{key}, with largest entry at most $\lfloor l/9\rfloor+1$, hence $$\mathcal{G}(\A_{3\times l})\leq\mathcal{G}(G)=\lceil\log_3(\lfloor l/9\rfloor+1)\rceil+3\leq\lceil\log_3(l+9)\rceil+1,$$ as desired.
\end{proof}

We use $E(i,j)$ to denote the matrix with $ij$-entry $1$ and all other entries 0.

\begin{prop}\label{k=6-7,3}
 Suppose that $k\in\{6,7\}$, that $q\geq 1$ and that $r \in\{3,k-3\}$. Then 
 $$\mathcal{G}(\A_{k\times l})\leq\mathcal{G}(\S_{k\times l})\leq\lceil \log_{k}(l+2k)\rceil +1.$$
\end{prop}

\begin{proof}
Let $v \in \Z^k$ have $r$ entries $(q+1)$ then $k-r$ entries $q$, and let 
$$N=\theta(v)-E(k-2,k-2)+E(k-2,k-1)+E(k-1,k-2)-E(k-1,k-1)$$ 
be a $(k,l)$-intersection array, so that $N^* = \{\{(q+2),(q+1)^{rk-2}, q^{k^2-rk+1}\}\}$. One can check computationally that $N$ is partwise fixed.  Lemma~\ref{minN} with $X=[(1,q+2)]$ shows that $$|G(N)/K_N|=|G(N)|=(q+2)!(q+1)!^{rk-2}q!^{k^2-rk+1}$$ is minimum amongst all $(k,l)$-intersection arrays with largest entry at least $q+2$. 

Let $A$ be a $(k,l)$-intersection array with $A^*=\{\{(q+1)^{rk},q^{(k-r)k}\}\}$, so that $(A-qJ)^\ast = \{\{1^{rk},$ $0^{(k-r)k}\}\}$.
Now \cite[Theorem 1.1]{ms} shows that $b(\S_{k\times r})>2$, so $|K_A|=|K_{A-qJ}|\geq 2$. Therefore,
$$|G(A)|\geq 2\cdot (q+1)!^{rk}q!^{(k-r)k} = \frac{2(q+1)}{q+2}|G(N)|>|G(N)|,$$
since $q\geq 1$. 

Thus, if  $A$ is a $(k, l)$-intersection array with $|G(A)|\leq |G(N)|$ then either $A\sim N$, or $\min A^\ast \le q-1$.
If $\min A^\ast \le q-1$, then from Lemma~\ref{minN} we deduce that $$|G(A)|\geq (q-1)!(q+1)!^{rk+1}q!^{(k-r)k-2} =\frac{(q+1)^2}{q(q+2)}|G(N)|>|G(N)|,$$ so $A\sim N$.

Therefore, since at least one of $q+1$ or $q$ is not divisible by $k$, and $\omega_N(q+1)$ and $\omega_N(q)$ are both at least $k$, we conclude that the greedy algorithm chooses an array satisfying the conditions of Lemma~\ref{key} with largest entry at most $\lfloor l/k\rfloor+2$, hence the result.
\end{proof}

The remaining proofs all follow a similar template. Namely, we exhibit a $(k,l)$-intersection array, $N$, yielding a small stabiliser, and through comparisons with $N$ we deduce that any array $A$ minimising $|G(A)|$ satisfies the conditions of Lemma~\ref{key} with some bounded largest entry. Throughout, we consider $A$ only up to $\sim$. 

\begin{prop}\label{other3}
    Suppose $k\in\{3,4\}$, and that $l=qk\pm1$ for some $q\geq 5$. Then $$\mathcal{G}(\A_{k\times l})\leq\mathcal{G}(\S_{k\times l})\leq \lceil\log_k(l+3k)\rceil+1.$$
\end{prop}
\begin{proof}
Let $\varepsilon=l-qk = \pm1$, and consider the $(k,l)$-intersection array $N$ given by $$\begin{bmatrix}
        q&q-\varepsilon&q+2\varepsilon\\
        q+\varepsilon&q+\varepsilon&q-\varepsilon\\
        q&q+\varepsilon&q\\
    \end{bmatrix} \text{ if $k=3$,} \quad \quad 
    \begin{bmatrix}
        q&q&q-\varepsilon&q+2\varepsilon\\
        q+\varepsilon&q&q+\varepsilon&q-\varepsilon\\
        q&q+\varepsilon&q&q\\
        q&q&q+\varepsilon&q\\
    \end{bmatrix} \text{ if $k=4$.}
$$

For $k = 4$ the third and fourth rows of $N$ have equal multisets of entries, as do the first and second columns, but otherwise the multisets of entries in all rows and columns are distinct. Hence \cite[Lemma~2.5]{cdvrd} shows that $K_N = 1$ if $k=3$ and $K_N\leq \langle (3\,4)\rangle\times\langle (1\,2)\rangle$ otherwise. Only the identity element of $\langle (3\,4)\rangle\times\langle (1\,2)\rangle$ preserves $N$, so $K_N = 1$ in all cases.

Let $A$ be a $(k,l)$-intersection 2-array with $|G(A)|\leq|G(N)|$. If $K_A \neq 1$ then Lemma~\ref{minN} shows that $|G(A)|\geq2\cdot (q+\varepsilon)!^kq!^{k^2-k}\geq\frac{2q(q-1)}{(q+1)^2}|G(N)|>|G(N)|$, since $q\geq 5$, a contradiction. Therefore $K_A = 1$. 
Without loss of generality assume that $a_{11}=\max A^*$ if $\varepsilon=1$, and $a_{11}=\min A^*$ otherwise. 

Suppose first that $a_{11}\geq q+2$ if $\varepsilon=1$, and $a_{11}\leq q-2$ if $\varepsilon=-1$. Then since $$\sum_{i,j\ne1} a_{ij}=kl-(l+l-a_{11})=(k-2)l+a_{11},$$
by applying~Lemma~\ref{minN} separately to the multisets $\{\{a_{ij} :i=1\text{ or }j=1\}\}$ and $\{\{a_{ij} : i,j\ne 1\}\}$ we deduce that $A\sim N$. Thus we may assume from now on that each extremal entry of $A$ is $q+\varepsilon$.

If $\omega_A(q+\varepsilon)<k$, then $A$ has a row with no $q + \varepsilon$ --- since the row sum is $qk+\varepsilon$,  this row contains an entry at least $q+2\varepsilon$ if $\varepsilon=1$ and at most $q+2\varepsilon$ if $\varepsilon=-1$, contradicting our assumption, so $\omega_A(q+ \varepsilon) \ge k$. 

Similarly, if $\omega_A(q)<k$ then without loss of generality the first row of $A$ is $v_1=(q+\varepsilon,q+\varepsilon,q-\varepsilon)$ if $k=3$ and $v_2=(q+\varepsilon,q+\varepsilon,q+\varepsilon,q-2\varepsilon)$ if $k=4$. 
If $k=3$, then one now can exhaustively check that $A=\theta(v_1)$, and thus is not partwise fixed, a contradiction, hence $\omega_A(q) \ge k$. If $k=4$ then $$A \mbox{ is } \theta(v_2) \mbox{ or } \begin{bmatrix}
q+\varepsilon&q+\varepsilon&q+\varepsilon&q-2\varepsilon\\
q&q+\varepsilon&q-\varepsilon&q+\varepsilon\\
q-\varepsilon&q&q+\varepsilon&q+\varepsilon\\
q+\varepsilon&q-\varepsilon&q&q+\varepsilon\\
\end{bmatrix},$$ neither of which are partwise fixed, a contradiction. Hence again $\omega_A(q) \ge k$. 

Therefore, in all cases $A$ satisfies the conditions of Lemma~\ref{key} with $t_2 \in \{q, q+\varepsilon\}$, and so the result will follow provided that $\max A^{*}\leq q+2$. This inequality has already been established above when $\varepsilon=1$, so suppose that $\varepsilon=-1$ and suppose, for contradiction, that $\max A^{*}> q+2$. Then by Lemma~\ref{minN}, $$|G(A)|\geq(q+3)!q!^{k^2-k-4}(q-1)!^{k+3}=\frac{(q+3)(q+2)(q-1)}{q^2(q+1)}|G(N)|>|G(N)|,$$ a contradiction, hence the result.
\end{proof}

We now present a technical lemma describing the structure of the intersection arrays which will yield a stabiliser of minimum size when $4\leq q\leq6$, $q\geq 11$, and $r=0$. 
In this and the next proof we shall use a $k \times k$ matrix $L(k)$ with first row $(q+2,q,q,\dots,q,q-1,q-1)$, second to penultimate rows equal to the first $(k-2)$ rows of 
$\theta((q-1,q+1,q,q,\dots,q))$, and last row $(q-1,q,q,\cdots,q,q+1)$, so that 
$L(k)^* = \{\{(q+2),(q+1)^{k-1},(q-1)^{k+1}, q^{k^2-2k - 1}\}\}.$ One may verify computationally for $4 \le k \le 6$ that $N$ is a partwise fixed $(k,l)$-intersection array.

\begin{lemma}\label{minA46}
Suppose $4\leq k\leq 6$, that $q\geq 11$, and that $r=0$. Let $A$ be a $(k,l)$-intersection $2$-array with $|G(A)|$ minimum. 
Then 
\begin{enumerate} 
\item[(i)] $A$ is partwise fixed, so has at most one row and one column $(q, q, \ldots, q)$, and $\max A^\ast \le q+3$; 
\item[(ii)] if $k$ divides every entry of $t$ of $A$ with $\omega_A(t)\geq k$, then $A$ has exactly $k-1$ entries less than $q$, at least one of which is  $q-2$. 
\end{enumerate}
\end{lemma}
\begin{proof}
Let $N = L(k)$. If $K_A \neq 1$ then $$|G(A)|\geq 2\cdot q!^{k^2}=\frac{2q^{k+1}}{(q+2)(q+1)^k}|G(N)|>|G(N)|,$$ since $q\geq 11$, a contradiction. Hence $K_A = 1$ and $A$ has at most one row and one column $(q, q, \ldots, q)$.
 
 If $\max A^*>q+3$, then Lemma~\ref{minN} shows that 
 $$|G(A)|\geq (q+4)!q!^{k^2-5}(q-1)!^{4}=\frac{(q+4)(q+3)q^{k-3}}{(q+1)^{k-1}}|G(N)|>|G(N)|,$$ 
 hence $\max A^*\leq q+3$. 
 
 It remains to show (ii), so suppose that $k$ divides every entry of $A$ that appears at least $k$ times. 
 
First suppose, for a contradiction, that $\omega_A(q)<k$. If a row of $A$ has a unique minimum entry $q-1$, then this row has exactly $k-2$ entries $q$, but $2(k-2) >\omega_A(q)$, so $A$ has \emph{at most} one such row. Thus at least $k-1$ rows of $A$ either have at least two $(q-1)$s, or have minimum entry less than $q-1$.
Hence there is some integer $0\leq s\leq k-1$ such that 
\begin{align*}
|G(A)|&\geq (q-2)!^s(q-1)!^{2(k-1-s)}(q+1)!^{2(k-1)}q!^{k^2-4k+4+s}\\
&\geq (q-1)!^{2(k-1)}(q+1)!^{2(k-1)}q!^{k^2-4k+4}
=\frac{(q+1)^{k-2}}{(q+2)q^{k-3}}|G(N)| >|G(N)|,
\end{align*} since $k\geq 4$, 
a contradiction. Therefore $\omega_A(q) \ge k$, and hence $k|q$. In particular $k$ does not divide $q-1$, whence $\omega_A(q-1)<k$.

Next, suppose for a contradiction that $\min A^*=q-1$. Then since $\omega_A(q-1)<k$, one (and hence by (i) exactly one) row has all entries $q$.  Therefore $\omega_A(q-1) = k-1$ and all other rows have a unique entry $q-1$. The same holds for columns, so without loss of generality $A$ has first row and column $(q, \ldots, q)$, and remaining entries forming a $(k-1,q(k-1))$-intersection array $B$ with $\max B^*=q+1$ and $\omega_B(q+1)=k-1$, with each occurrence of $q+1$ in a distinct row and column. Thus, by~\cite[Lemma 2.7]{cdvrd} $B$ is not partwise fixed, hence neither is $A$, a contradiction. Therefore $\min A^*<q-1$.

Since $A$ has at most one row with all entries $q$, the matrix $A$ has at least $k-1$ entries less than $q$. If $\min A^* < q-2$ then 
$$|G(A)|\geq (q-3)!(q-1)!^{k-2}(q+1)^{k+1}q!^{k^2-2k}=\frac{q^2(q+1)}{(q+2)(q-1)(q-2)}|G(N)|>|G(N)|,$$ 
hence $\min A^* =q-2$. If $A$ has more than $k-1$ entries less than $q$, then 
\[|G(A)|\geq (q-2)!(q-1)!^{k-1}(q+1)!^{k+1}q!^{k^2-2k-1}=\frac{q(q+1)}{(q-1)(q+2)}|G(N)|>|G(N)|.\] 
Therefore $A$ has exactly $k-1$ entries less than $q$, as was to be shown.
\end{proof}

\begin{prop}\label{k=4-6,0}
    Suppose that $4\leq k\leq 6$, that $q\geq 11$, and that $r=0$. Then $$\mathcal{G}(\A_{k\times l})\leq\mathcal{G}(\S_{k\times l})\leq\lceil \log_{k}(l+3k)\rceil +1.$$
\end{prop}

\begin{proof}
Let $A$ be a $(k,l)$-intersection 2-array. If $A$ minimises $|G(A)|$ then Lemma~\ref{minA46} holds for $A$. If in addition $A$ satisfies the conditions of Lemma~\ref{key} then it does so with largest entry at most $(l/3)+3$, and the result follows from a short calculation. 

Assume therefore, for a contradiction, that $|G(A)|$ is minimal but $A$ fails the conditions of Lemma~\ref{key}. Recall the definition of $L(k)$ from just before Lemma~\ref{minA46}. 

\medskip
\noindent \underline{\textbf{Case $k\in\{5,6\}$:}}
Let $$N=\begin{pmatrix}
q&\aug&qJ_{1\times(k-1)}\\\hline
qJ_{(k-1)\times 1}&\aug&L(k-1)\\
\end{pmatrix}.$$ Then $N$ is a partwise fixed $(q,kq)$-intersection array, and Lemmas~\ref{minA46}(ii) and~\ref{minN} show that $$|G(A)|\geq (q-2)!(q-1)!^{k-2}(q+1)!^kq!^{k^2-2k+1}=\frac{q(q+1)}{(q+2)(q-1)}|G(N)|>|G(N)|,$$ giving the required contradiction.

\medskip
\noindent \underline{\textbf{Case $k=4$:}}
If $A$ has two entries $q-2$ then by Lemmas~\ref{minA46}(ii) and~\ref{minN} 
$$|G(A)|\geq (q-2)!^2(q-1)!(q+1)!^5q!^8=\frac{q^2(q+1)}{(q+2)(q-1)^2}|G(L(4))|>|G(L(4))|,$$ a contradiction. Since $A$ has exactly $k-1=3$ entries less than $q$, these are in distinct rows and columns, otherwise $A$ would have two rows or columns $(q,\dots,q)$. 
Thus there exist $a_{ij} \ge q$ such that $$A=\begin{bmatrix}q&q&q&q\\
q&q-2&a_{23}&a_{24}\\
q&a_{32}&q-1&a_{34}\\
q&a_{42}&a_{43}&q-1\\
\end{bmatrix}.$$ Each row and column sums to $4q$, so $a_{42}+a_{32}=a_{42}+a_{43}+1$, so $a_{32}=a_{43}+1$ and therefore $a_{34}+a_{43}=a_{34}+(a_{32}-1)=2q$. Therefore $a_{34}=a_{43}=q$, hence $a_{23}=a_{24}=a_{32}=a_{42}=q+1$, so $((3\,4),(3\,4))\in K_A$, contradicting $K_A  =1$.\end{proof}

\begin{prop}\label{k=4-6,2}
    Suppose $4\leq k\leq 6$, that $q\geq 4$, and that $r\in\{2,k-2\}$. Then $$\mathcal{G}(\A_{k\times l})\leq\mathcal{G}(\S_{k\times l})\leq\lceil \log_{k}(l+3k)\rceil +1.$$
\end{prop}

\begin{proof}
Let $v \in \mathbb{Z}^k$ have $r$ entries $q+1$ then $k-r$ entries $q$, and let $N$ be the $(k, l)$-intersection array $\theta(v)+E(1,1)-E(1,2)-E(2,1)+E(2,2)$, so that $N^\ast = \{\{(q+2)^2, (q+1)^{rk-3}, q^{k^2-rk},q-1\}\}$. We verify computationally that $K_N = 1$. 

Let $A$ be a $(k,l)$-intersection 2-array with $|G(A)|\leq |G(N)|$. By Lemma~\ref{minN}, if $K_A \neq 1$ then 
$$|G(A)|\geq 2\cdot (q+1)!^{rk}q!^{(k-r)k}=\frac{2q(q+1)}{(q+2)^2}|G(N)|> |G(N)|,$$ 
since $q\geq 4$, a contradiction. Thus $K_A = 1$.  If $\max A^*\geq q+4$ then 
$$|G(A)|\geq (q+4)!(q+1)!^{kr-4}q!^{k(k-r)+3}=\frac{q(q+4)(q+3)}{(q+2)(q+1)^2}|G(N)|>|G(N)|,$$ 
hence $\max A^*\leq q+3$. Therefore, by Lemma~\ref{key} and a short calculation it suffices to show that $\omega_A(t)\geq k$ for some $t$ that is not divisible by $k$. We shall show that both $\omega_A(q+1)$ and $\omega_A(q)$ are at least $k$, from which the result will follow.

\medskip

\noindent \textbf{\underline{$\omega_A(q+1)$:}} Assume for a contradiction that $\omega_A(q+1)<k$. Since $k \ge 4$, this implies that at least three rows have less than two $(q+1)$s, without loss of generality the first three rows, and the first row has no $(q+1)$s. To make the row sum $l$, each of the first three rows of $A$ have an entry greater than $q+1$. If one of these entries is greater than $q+2$, then $$|G(A)|\geq (q+3)!(q+2)!^2(q+1)!^{kr-7}q!^{(k-r)k+4}=\frac{q(q+2)(q+3)}{(q+1)^3}|G(N)|>|G(N)|,$$ a contradiction, hence 
these largest entries are all $q+2$. If $\min A^*<q$, then $|G(A)|\geq (q+2)!^3(q+1)!^{kr-5}q!^{(k-r)k+1}(q-1)!=\frac{q+2}{q+1}|G(N)|>|G(N)|,$ a contradiction. Therefore, $\min A^* \ge q$. 

If $r=2$ these conditions imply that $$A=\begin{pmatrix}
2I_{3}+qJ_{3 \times 3}&\aug&qJ_{3 \times (k-3)}\\\hline
qJ_{(k-3) \times 3}&\aug&B\\
\end{pmatrix},$$ where $B$ is a $(k-3,(k-3)q+2)$-intersection array. But then $((1\,2),(1\,2))\in K_A$, a contradiction. 

We may therefore assume from now that $r=k-2\geq 3$. By assumption, the first row of $A$ has no $q+1$, so since it has greatest entry $q+2$ it must have at least two such entries. Additionally, the second and third rows of $A$ each have at least one entry $q+2$. 
Therefore, $A$ has at least four entries greater than $q+1$. Lemma~\ref{minN} now shows
$$|G(A)|\geq(q+2)!^4(q+1)!^{kr-8}q!^{(k-r)k+4}=\frac{q(q+2)^2}{(q+1)^3}|G(N)|>|G(N)|,$$ 
a contradiction. Therefore $\omega_A(q+1)\geq k$.

\medskip

\noindent \textbf{\underline{$\omega_A(q)$:}} Assume for a contradiction that $\omega_A(q)<k$. Then analogously to the previous case, the first three rows of $A$ each have at most one $q$, hence $A$ has at least three entries less than $q$. Therefore, $$|G(A)|\geq (q+1)!^{kr+3}q!^{(k-r)k-6}(q-1)!^{3}=\frac{(q+1)^4}{q^2(q+2)^2}|G(N)|>|G(N)|,$$ giving the required contradiction.
\end{proof}

Before our final proposition we describe the structure of the 2-arrays with minimum stabilisers.

\begin{lemma}\label{minA56}
Suppose that $5\leq k\leq6$, that $q\geq7$, and that $r=\pm1$. Let $A$ be a $(k,l)$-intersection $2$-array such that $|G(A)|$ is minimum. Then 
\begin{enumerate}
\item[(i)]$A$ is partwise fixed and $\max A^*\leq \lfloor l/k\rfloor +3$; and 
\item[(ii)] if $k$ divides every entry $t$ of $A$ with $\omega_A(t) \ge k$, then $\omega_A(q+sr)=0$ for $s\geq3$,
$\omega_A(q+ r)=k-1$, and $\omega_A(q+2r)=2$ with both such entries in distinct rows and columns.
\end{enumerate}
\end{lemma}
\begin{proof}
Let $v = (q,q,\dots,q,q+r)$ and let
$$X=E(1,k)-E(1,1)-E(k,k)+E(k,1)+E(2,k-1)-E(2,k-2)+E(3,k-2)-E(3,k-1).$$ Let $N$ be the $(k, l)$-intersection array $\theta(v)+ rX$, so that $N^*=\{\{q+2r,(q+r)^{k+2},q^{k^2-k-7},(q-r)^4\}\}$: we check computationally that $K_N = 1$. 

If $K_A\ne1$ then $$|G(A)|\geq2\cdot (q+ r)!^kq!^{k^2-k}\geq\frac{2q^3(q-1)}{(q+1)^4}|G(N)|>|G(N)|,$$ since $q\geq 7$, a contradiction, hence $K_A$ is partwise fixed. 
Moreover, if $\max A^*>\lfloor l/k\rfloor +3$ then \begin{align*}|G(A)|&\geq\begin{cases}
    (q+4)!(q+1)!^{k-4}q!^{k^2-k+3}&\text{ if $ r=1$}\\
    (q+3)!(q-1)!^{k+3}q!^{k^2-k-4}&\text{ if $ r=-1$}\end{cases}
& \geq \frac{q^4(q+4)(q+3)}{(q+1)^6}|G(N)| > |G(N)|,\end{align*} so $\max A^*\leq \lfloor l/k\rfloor +3$, proving (i).

For Part (ii), suppose that $k$ divides every entry $t$ of $A$ with $\omega_A(t) \ge k$. Thus either $\omega_A(q)<k$ or $\omega_A(q+ r)<k$. Seeking a contradiction, suppose that $\omega_A(q) < k$.  We shall say `$a$ is more extreme than $b$' to mean $a>b$ when $ r=1$ and $a<b$ when $ r=-1$ --- `less extreme than' is defined analogously. Since $\omega_A(q)<k$ and $k\geq 5$ all but at most two rows contain either two $(q- r)$s 
or an entry less extreme than $q- r$, hence there is some $s \in \{0, \ldots, k-2\}$ 
such that 
$$|G(A)|\geq (q-2 r)!^s(q- r)!^{2(k-2-s)}(q+ r)!^{3k-4}q!^{k^2-5k+8+s}.$$ 
A short calculation shows that the right hand side is minimised by taking 
$s=0$, thus 
$$|G(A)|\geq(q- r)!^{2k-4}(q+ r)!^{3k-4}q!^{k^2-5k+8}\geq\frac{(q+1)^{2k-8}(q-1)}{q^{2k-7}}|G(N)|>|G(N)|,$$ 
where the second inequality follows by considering both cases for $ r$. 
This is a contradiction, hence $\omega_A(q)\geq k$ and so $\omega_A(q+ r)<k$.
Every row and column of $A$ contains at least one entry more extreme than $q$, hence without loss of generality $a_{11}$ is more extreme than $q+ r$ since $\omega_A(q+ r)<k$. 

Suppose for a contradiction that no  row other than the first contains an entry more extreme than $q+r$. Then all other rows and columns contain exactly one entry more extreme than $q$, and moreover this must be $q+ r$. Therefore, rows $2$ to $k$ of $A$ are permutations of $v$, a contradiction as the first column of $A$ must have an entry less extreme than $q$. Similarly for columns, thus $A$ has at least two entries more extreme than $q+ r$, and these entries occur in distinct rows and columns. 

If one of these entries is more extreme than $q+2r$ then checking both possibilities for $ r$ shows that $$|G(A)|\geq (q+3 r)!(q+2 r)!(q+ r)!^{k-2}(q- r)!^{3}q!^{k^2-k-3}\geq\frac{(q+3)(q+2)q}{(q+1)^3}|G(N)|>|G(N)|,$$ a contradiction, hence $A$ contains at least two $(q+2 r)$s. 
Moreover, there cannot be more than two $(q + 2  r)$s, as $$(q+2 r)!^3(q+ r)!^{k-3}(q- r)!^3q!^{k^2-k-3}\geq\frac{q(q+2)^2}{(q+1)^3}|G(N)|>|G(N)|.$$ Thus $\omega_A(q+2r) = 2$.

Finally, if $\omega_A(q+ r)  \leq k-2$, then our conditions imply that $$A=\begin{pmatrix}
(q- r)J_{2\times 2}+3 r I_{2\times 2}&\aug&qJ_{2\times(k-2)}\\\hline
qJ_{(k-2)\times 2}&\aug&qJ_{(k-2)\times(k-2)}+  r I_{(k-2)\times(k-2)}\\
\end{pmatrix},$$ which is not partwise fixed, whence $\omega_A(q+ r)=k-1$, as was to be shown.
\end{proof}

We now wrap up with our final proposition.

\begin{prop}\label{k=5-6,1}
Suppose $5\leq k\leq6$, that $q\geq7$, and that $r=\pm1$. Then $$\mathcal{G}(\A_{k\times l})\leq\mathcal{G}(\S_{k\times l})\leq \lceil\log_{k}(l+3k)\rceil+1.$$
\end{prop}
\begin{proof}
Let $A$ be a $(k,l)$-intersection 2-array minimising $|G(A)|$. By Lemma~\ref{minA56}(i) the group $K_A$ is trivial, and $\max A^\ast \le q+3$ if $ r = 1$ and $q+2$ otherwise. 
Hence if  there exists a $t$ that is not divisible by $k$ and satisfies $\omega_A(t) \ge k$ , then we are done by 
Lemma~\ref{key}, so suppose that this does not hold. Now Lemma~\ref{minA56}(ii) shows that $\omega_A(q+2 r) =2$, 
$\omega_A(q+ r) = k-1$, and if $r = 1$ then $\max A^\ast \le q+2$, while if $r = -1$ then $\min A^\ast \ge q-2$. 
\medskip

\noindent \underline{\textbf{Case $k=6$:}} Here $A$ has exactly one row and one column with more than one entry in $\{q+ r, q+2 r\}$.
Thus, without loss of generality the first three rows of $A$ have entry multiset $T=\{\{q^{k-1},q+ r\}\}$, let these entries $q+ r$ be $a_{1j_1}$, $a_{2j_2}$, and $a_{3j_3}$. 
Note that $j_1$, $j_2$, and $j_3$ are distinct, for if not then $A$ has two identical rows, contradicting $K_A = 1$, so without loss of generality $j_i = i$. Since the two $(q+2r)$s are in distinct columns, without loss of generality columns $1$ and $2$ have entry multiset $T$, for if not then $A$ has at most $6-2-2=2$ columns with entry multiset $T$, a contradiction. But then $((1\,2),(1\, 2))\in K_A$, contradicting $K_A = 1$.

\medskip

\noindent \underline{\textbf{Case $k=5$:}} From Lemmas~\ref{minN} and \ref{minA56}, $K_A = 1$ and with some work we see that up to $\sim$
$$A=\begin{bmatrix}
q+2 r&q&q- r&q&q\\
q- r&q+2 r&q&q&q\\
q&q&q+ r&q&q\\
q&q- r&q+ r&q+ r&q\\
q&q&q&q&q+ r
\end{bmatrix}.$$

Since by assumption $5$ divides every entry of $A$ which appears with multiplicity at least $5$, we deduce that $q=5m$ for some $m$.
Let $W\in(\Z^k)^{k\times k}$ be any $(k,l)$-intersection 3-array arising from $A$ which has one entry each of \begin{align*}&(m+ r,m,m,m,m),(m,m+ r,m,m,m),(m,m,m+ r,m,m),\\ &(m,m,m,m+ r,m),(m+ r,m,m,m,m+ r),(m,m+ r,m+ r,m,m),\\&(m- r,m,m,m,m),(m,m- r,m,m,m),(m,m,m- r,m,m),\end{align*} and all other entries $(m,m,m,m,m)$ --- this is possible since the sum in each coordinate is $25m+ r=5q+ r$. Since $A$ is partwise fixed, $K_W$ acts diagonally as a subgroup of $\S_5$ on each of the entries of $W$, but the first four tuples as above yield a trivial pointwise stabiliser, hence $K_W = 1$. Now $W$ satisfies the assumptions of Lemma~\ref{key}, with at least one of $m$, $m+r$ not divisible by $5$ and maximum entry $m+1 = q/k+1$, so the result follows.
\end{proof}

\begin{table}[h]
\centering
\begin{tabular}{lll|lll}
$k$ & $l$ & ref.&$k$ & $l$ & ref.\\
\hline
$2$&$l\geq 10$&Proposition~\ref{k=2}&6&$6q$, $q\geq 11$&Proposition~\ref{k=4-6,0} \\
&&&&$6q\pm2$, $q\geq 5$&Proposition~\ref{k=4-6,2}\\

$3$&$3q$, $q\geq 5$&Proposition~\ref{k=3}&&$6q\pm1$, $q\geq 7$&Proposition~\ref{k=5-6,1}\\
&$3q\pm 1$, $q\geq5$&Proposition~\ref{other3}&&$6q\pm3$, $q\geq 1$&Proposition~\ref{k=6-7,3}\\
&&&\\ 

$4$&$4q$, $q\geq 11$ &Proposition~\ref{k=4-6,0}  &7&$7q+r$, $q \ge 1$, $-2\leq r\leq 2$&Proposition~\ref{logbase}\\
&$4q\pm 1$, $q\geq5$ & Proposition~\ref{other3}&&$7q\pm3$, $q\geq 1$&Proposition~\ref{k=6-7,3}\\
&$4q+2$, $q\geq 4$&Proposition~\ref{k=4-6,2}&&\\
&&&$\geq 8$& $2$&Proposition~\ref{l=2}\\

5&$5q$, $q\geq 11$&Proposition~\ref{k=4-6,0}& & $\ge 3$ & Proposition~\ref{logbase} \\
&$5q\pm2$, $q\geq 5$&Proposition~\ref{k=4-6,2}&  \\
&$5q\pm1$, $q\geq 7$&Proposition~\ref{k=5-6,1}
\end{tabular}
\caption{A summary of where each pair $(k,l)$ is considered.} \label{tab1}
\end{table}

\begin{thm}\label{camskl}
Let $\mathcal{F}_2=\{\A_{k\times l},\S_{k\times l}: k,l\geq 2\text{ with } kl >4\}$. Then $\mathcal{F}_2$ is 11-ravenous.
\end{thm}
\begin{proof}
Let $G\in \mathcal{F}_2$. Morris and Spiga show in~\cite[Theorems 1.1 and 1.2]{ms} that $b(G)\geq\log_k(l+2)$.  In Table~\ref{tab1} we summarise where each pair $(k, l)$ has been considered so far. By comparing $\log_k(l+2)$ with the bounds on $\mathcal{G}(G)$ in each relevant proposition, we find that if $(k,l)$ is listed in Table~\ref{tab1} then $\mathcal{G}(G)/b(G) \le 11$. 
For the finitely many $(k,l)$ not appearing in Table~\ref{tab1}, let $n = |\Pi_{k,l}| = ((kl)!)/(l!^kk!)$. From~\cite[Theorem 4.4]{blaha}, $\mathcal{G}(G)/b(G)\leq
\log(\log n) +1.$ Among the remaining $(k,l)$, this is maximised when $k=6$ and $l=60$. In this case we obtain $\mathcal{G}(G)/b(G)\leq 11$, as desired.
\end{proof}


\begin{thebibliography}{99}

\bibitem{blaha}
K.D. Blaha, \emph{Minimum bases for permutation groups: the greedy approximation}, J. Algorithms \textbf{13(2)}
(1992), 297--306.

\bibitem{soluble}
S. Brenner, C. del Valle, C.M. Roney-Dougal, \emph{Irredundant bases for soluble groups}, arXiv preprint: \url{https://arxiv.org/abs/2501.03003}, (2025).
\bibitem{burn2}
T.C. Burness, R.M. Guralnick, J. Saxl, \emph{On base sizes for symmetric groups}, Bull. Lond. Math. Soc. \textbf{43(2)} (2011), 386--391.
\bibitem{bgl}
T.C. Burness, M. Garonzi, A. Lucchini, \emph{Finite groups, minimal bases and the intersection number}, Trans. London Math. Soc. \textbf{9} (2022), 20--55.
\bibitem{cam}
P.J. Cameron, \emph{Permutation Groups}, London Mathematical Society Student Texts \textbf{45}, Cambridge University Press, (1999).
\bibitem{cdvrd}
P.J. Cameron, C. del Valle, C.M. Roney-Dougal, \emph{Regular bipartite multigraphs have many (but not too many) symmetries}, Discrete Anal., to appear.

\bibitem{camsol}
P.J. Cameron, R. Solomon, A. Turull, \emph{Chains of subgroups in symmetric groups}, J. Algebra \textbf{127(2)} (1989), 340--352.

\bibitem{sporadic}
C. del Valle, \emph{Greedy base sizes for sporadic simple groups}, J. Group Theory, (2025).
\bibitem{dvrd} C. del Valle, C.M. Roney-Dougal, \emph{The base size of the symmetric group acting on subsets}, Algebr. Comb. \textbf{7(4)} (2024), 959--967.

\bibitem{james}
J.P. James, \emph{Partition actions of symmetric groups and regular bipartite graphs}, Bull. London Math. Soc. \textbf{38(2)} (2006), 224--232.
\bibitem{Lee}
M. Lee, \emph{Primitive almost simple IBIS groups with sporadic socle}, arXiv preprint: \url{https://arxiv.org/abs/2302.01521}, (2023).
\bibitem{ma}
T. Maund, \emph{Bases for permutation groups}, Ph.D. thesis, University of Oxford, (1989).

\bibitem{MeSp}
G. Mecenero, P. Spiga, \emph{A formula for the base size of the symmetric group in its action on subsets}, Australas. J. Combin. \textbf{88(2)} (2024), 244--255.

\bibitem{ms}
J. Morris, P. Spiga, \emph{On the base size of the symmetric and the alternating group acting on partitions}, J. Algebra \textbf{587} (2021), 569--593.

\bibitem{peiran}
C.M. Roney-Dougal, P. Wu, \emph{Irredundant bases for the symmetric group}, Bull. Lond. Math. Soc. \textbf{56(5)} (2024), 1788--1802

\bibitem{ser}
\'A. Seress, \emph{Permutation Group Algorithms}, Cambridge Tracts in Mathematics \textbf{152}, Cambridge University Press
(2003).
\end{thebibliography}
\end{document}